\documentclass[notitlepage]{report}
\usepackage{amsmath, amssymb, amsthm, relsize, tikz-cd}

\theoremstyle{plain}
\newtheorem{theorem}{Theorem}[section]
\newtheorem{lemma}[theorem]{Lemma}

\newtheorem{dfn}[theorem]{Definition}

\theoremstyle{remark}
\newtheorem{remark}{Remark}

\DeclareMathOperator*{\OmSum}{\mathlarger{\mathlarger{\mathlarger{ \Omega}}}}
\DeclareMathOperator*{\MhSum}{\mathlarger{\mathlarger{\mathlarger{\mho}}}}
\DeclareMathOperator{\len}{len}
\DeclareMathOperator{\Res}{Res}
\DeclareMathOperator{\Rsd}{Rsd}

\begin{document}

\title{The Compositional Integral\\The Narrow And The Complex Looking-Glass}

\author{James David Nixon\\
	JmsNxn92@gmail.com}

\maketitle

\begin{abstract}
The goal of this paper is to formalize The Compositional Integral in The Complex Plane. We prove a convergence theorem guaranteeing its existence. We prove an analogue of Cauchy's Integral Theorem--and suggest an approach at recovering Cauchy's Integral Formula. With this we derive a modified form of Cauchy's Residue Theorem. Then, we develop a compositional analogue of Taylor Series. In finality, we describe a compositional Fourier Transform; and illustrate some basic properties of it.
\end{abstract}



\tableofcontents

\chapter{A Brief Overview}\label{chap1}

\section{Preface}\label{secO1}
\setcounter{equation}{0}

This section is intended to remind the reader of what the notation we use throughout denotes. This notation is largely novel, and to a seasoned mathematician may seem a \textit{tad odd}. A large portion of this work speaks of First Order Differential Equations; but in almost its entirety, we make no reference to existing work. This is quite frankly because the work requires we speak differently than how convention dictates we speak. And largely because the author does not like the state of affairs.

Of these notations, and these notions; they begin by being rather intuitive and simple. They appear almost as inconsequential results of the grander theory of differential equations as they exist currently. But in using these results to rebuild the theory we can speak differently; and introduce new and exciting ideas. 

The majority of this work will be rather simple. This is derided by the fact, that largely this work is notational. And even, more largely, an exact analogue of what is usually found in the field of analysis. As per this, this paper can be more or less considered a motivation for a system of notations. A formalist's dream so to speak.

The exception of this work, and in the rephrasal of everything; is that calculus is not treated as a sum and product game. Where in the usual basis of calculus, next to everything is sums and products and limits of these things. In our case, the central object of concern is the composition operator. As opposed to infinite sums, and infinite products, we have infinite compositions. As opposed to integrals and derivatives, we have The Compositional Integral and First Order Differential Equations.

The point of this first chapter is to mostly recant work done in \cite{Nix, Nix2}. For a more detailed exposition on the work done here, the reader is asked to refer to these papers. Upon which, most of the work is open ended and, encore un esprit de formalit\'{e}.

\section{The $\Omega$-notation}\label{secO2}
\setcounter{equation}{0}

In the paper \cite{Nix2}, the author developed an approach at handling infinite compositions. It is a long paper, and to expect the author to summarize it is a bit absurd. Nonetheless, not much is used from this paper except the notation. But, the brunt of the motivation for this notation comes from this paper. We recommend you read \cite{Nix2}.

Throughout this paper the symbols $\OmSum$ and $\MhSum$ shall be reserved for use as complimentary operators. If $\phi_j$ is a sequence of functions then,

\[
\OmSum_{j=n}^m \phi_j(z) \bullet z = \phi_n(\phi_{n+1}(...\phi_m(z)))
\]

And,

\[
\MhSum_{j=n}^m \phi_j(z) \bullet z = \phi_m(\phi_{m-1}(...\phi_n(z))))
\]

We can refer to operations across $\OmSum$ as inner compositions; and complimentary, operations across $\MhSum$ as outer compositions. This language is somewhat novel, and arises from adding terms on the inside or on the outside. Sadly the author could not find the origination of these terms. Or where he found them. The standard literature calls these right handed or left handed compositions, respectfully. Luckily we are in a position of standardization, as there is no true standard.

Here the $\bullet\, z$ represents which variable we perform our compositions across. In such a sense, $\bullet$ binds $z$ to $\OmSum$ (or $\MhSum$). This becomes necessary when our functions $\phi_j$ depend on some other paramater $s$. Which is to mean,

\[
\OmSum_{j=n}^m \phi_j(s,z)\bullet z = \phi_n(s,\phi_{n+1}(s,...\phi_m(s,z)))
\]

And,

\[
\OmSum_{j=n}^m \phi_j(s,z)\bullet s = \phi_n(\phi_{n+1}(...\phi_m(s,z),...z),z)
\]

This follows similarly for the operator $\MhSum$. The indication of these notations is rather straightforward; they behave little differently than the notation $\sum$ or $\prod$. Except, it is necessary we bind them to a variable.

These two operators are complimentary and they serve to describe orientation. Which is whether we have left handed orientation or right handed orientation. These two orientations are a tad bit perverse compared to what we usually think about when we posit orientation; they are related by the functional inverse, and they're non-abelian. If we take $f^{-1}$ to represent the functional inverse, this means,

\[
\Big{(}\OmSum_{j=n}^m \phi_j(z)\bullet z\Big{)}^{-1} = \MhSum_{j=n}^m \phi_j^{-1}(z)\bullet z
\]

This relationship is of dire importance. It is necessitated that orientation enters the conversation every so often; but it will be slightly different from what we're used to. In that sense, $\MhSum$ is the inverse orientation of $\OmSum$. It's pretty easy to remember, they are flipped versions of the same symbol. 

The author chose $\OmSum$, or inner compositions, to be the canonical orientation; but in truth there is no correct orientation. Ironically, $\OmSum$ serves to be the more difficult case, and also serves to be the more useful case in all the instances the author has encountered these compositions. On occasion, we may see an instance where $\MhSum$ appears more naturally--but on the whole we will reserve the use of this symbol to when necessary. With that, we can think of $\OmSum$ as forwards, and $\MhSum$ as backwards. This is also easy to remember because the indexes go forwards in $\OmSum$, and go backwards in $\MhSum$. And $\mho$ is upside down; indicating backwards.

The bullet $\bullet$ also serves an additional purpose. It is designed to act as a non-commutative product a tad different than the typical $\circ$-notation. If the author is to write $f \bullet g \bullet z$, it is intended to mean $f(g(z))$. This notational stitch becomes very convenient when $f$ and $g$ depend on other variables and it is difficult to denote composition. It also, in this sense, can be thought similarly to a differential form; where $\OmSum_j f_j \bullet g_j \bullet z$ has clear meaning as $\OmSum_j f_j(g_j(z))\bullet z$. Of the same character, $\OmSum_j f_j \bullet g \bullet z$ will be given the meaning $\big{(}\OmSum_j f_j \bullet z\big{)}\bullet g \bullet z$. The $\bullet$-notation is rather fluid but the reader should care to notice that without it this paper would probably stretch ten more pages. We will only use this convenience when necessary, but it certainly has its advantages. And upon usage of these symbols a clear description will be inferred by context.

As is the case with traditional analysis--the central point of study is when we let $m\to\infty$. Expressions of these forms shall loosely be referred to as Infinite Compositions. With respect to this, the most important aspect of this notation can be summarized by the following theorem. This theorem was partially presented in \cite{Nix2}. There we spoke more fluidly about the types of situations where one can get theorems like the below. Here we will only state what is required for this exposition.

\begin{theorem}[The Compactly Normal Convergence Theorem]\label{ThmNormCvg}
Let $\mathcal{S}$ and $\mathcal{G}$ be domains in $\mathbb{C}$. Suppose $\{\phi_j(s,z)\}_{j=0}^\infty$ is a sequence of holomorphic functions such that $\phi_j : \mathcal{S} \times \mathcal{G} \to \mathcal{G}$. If for all compact disks $\mathcal{B} \subset \mathcal{S}$ and $\mathcal{K} \subset \mathcal{G}$ the following sum converges,

\[
\sum_{j=0}^\infty ||\phi_j(s,z) - z ||_{\mathcal{B},\mathcal{K}} < \infty
\]

Then the infinite compositions,

\[
\OmSum_{j=0}^\infty \phi_j(s,z)\bullet z
\]

And,

\[
\MhSum_{j=0}^\infty \phi_j(s,z)\bullet z
\]

Converge uniformly on all compact sets $\mathcal{B} \subset \mathcal{S}$ and $\mathcal{K} \subset \mathcal{G}$.
\end{theorem}

This theorem tells us we have good control over the infinite compositions if they are compactly normally summable. The reader may also benefit from the intuition that for every compact set $\mathcal{K}$ there is a larger compact set $\mathcal{L}$ such that,

\[
||\OmSum_{j=n}^m \phi_j(s,z)\bullet z - z||_{\mathcal{B},\mathcal{K}} \le \sum_{j=n}^m ||\phi_j(s,z) - z||_{\mathcal{B},\mathcal{L}}
\]

This comparison underlies the entire theory of this paper. In controlling the sum on the right hand side we can control the expression on the left. It has just enough malleability to make the transitition back and forth reasonable.

The important thing to remember from the phrasing of this theorem; is if the sequence of functions are normally summable, they are normally composable. Furthering this, limits behave just as well by comparing them to limits of sums. The interchange of sums to compositions and back will be done frequently in this paper. It is important to be reminded that this is done entirely rigorously; largely in due part to our work done in \cite{Nix2}.

\section{The differential bullet product $ds\bullet z$}\label{secO3}
\setcounter{equation}{0}

The second thing to introduce is the differential bullet product. This notion was more clearly developed in \cite{Nix}. Here we will simply carve out some key properties of it. The author strongly recommends reading \cite{Nix}, as we will make some logical leaps in this paper; of which are more straight and forward if you read \cite{Nix}.

The differential bullet product $ds \bullet z$ is intended to be paired with an integral $\int$. Nonetheless, treating it as a type of differential form can help build intuition, and further enlighten a more general character. The bullet $\bullet$ becomes integrally connected to the above $\Omega$-notation.

To begin, if we call $y$ the solution to the equation,

\[
y(x) = z + \int_{a}^x \phi(s,y(s))\,ds
\]

Then this can be written more compactly,

\[
y(x) = \int_a^x \phi(s,z)\,ds\bullet z
\]

The reason for the bullet aligns with our $\Omega$-notation from above. Using Euler's method, if $\{s_j\}_{j=0}^n$ is a partition of $[a,b]$ in \emph{descending} order, and supposing $s_{j+1} \le s_j^* \le s_j$ and $\Delta s_j = s_j - s_{j+1}$ then,

\[
\int_a^b \phi(s,z)\,ds\bullet z = \lim_{\Delta s_j \to 0} \OmSum_{j=0}^{n-1}z + \phi(s_j^*,z)\Delta s_j \bullet z 
\]

The $\Delta s_j$ term can be thought of as an infinitesimal increment which looks like $ds$ in the limit. The expression $\OmSum_{j=0}^{n-1} z+$ is a fancy kind of $\sum_{j=0}^{n-1}$ which in the limit becomes $\int_a^b$. Which is a rough manner of understanding the notation.

Mostly out of preference the author chooses to use the operator $\OmSum$ and forces our partition to be descending. If we chose an ascending partition $a_j = s_{n-j}$; then the expression could also be written,

\[
\MhSum_{j=0}^{n-1} z + \phi(a_j^*,z)\Delta a_j \bullet z
\]

The author purposefully avoids the operator $\MhSum$ as much as possible. This is done mostly to maintain clarity and the greater importance of $\OmSum$. Though, in practise, both expressions are equivalent. He imagines if he used both symbols to their extreme it may be a tad confusing.

This differential form also satisfies the usual laws of Leibniz substitution. Where if $s = \gamma(u)$ and $ds = \gamma'(u)du$, then $ds \bullet z = \gamma'(u)du\bullet z$. Or written under the integral, if $\gamma(a') = a$ and $\gamma(b') = b$,

\[
\int_a^b \phi(s,z)\,ds \bullet z = \int_{a'}^{b'} \phi(\gamma(u),z)\gamma'(u)\,du \bullet z
\]

Symbolically, all that's needed is a quick application of the mean value theorem to get $\Delta s = \gamma' \Delta u$. Then the composition behaves no different. This again gives us a glimmer of the concept of orientation. If we take the integral from $b$ to $a$ instead,

\[
\int_b^a\phi(s,z)\,ds\bullet z = \MhSum_{j=0}^{n-1} z - \phi(s_j^*,z)\Delta s_j \bullet z
\]

Where now everything is backwards compositionally. This is especially true because $z - \phi(s_j^*,z)\Delta s_j \approx \big{(}z + \phi(s_j^*,z)\Delta s_j\big{)}^{-1}$, which is the functional inversion. This statement is correct (though it needs to be stated with some caveats) and aligns perfectly with the crude statement $\int_a^b = \big{(}\int_{b}^a\big{)}^{-1}$. In such a sense, our sense of orientation is compatible with integration.

We refer to the coupled pair $\int ... ds \bullet z$ as The Compositional Integral. Its similarity to the usual integral extends in many manners. Its main exception is that it behaves under composition as the integral behaves under addition. And the whole mess is non-abelian. So, there's sure to be some trade-offs in the switch up. The most striking resemblence being,

\[
\int_b^c \phi(s,z)\,ds\bullet \int_a^b \phi(s,z)\,ds\bullet z = \int_a^c \phi(s,z)\,ds \bullet z
\]

Which keeps in tone with the bullet notation above. To further our notational conveniences; we will sometimes write expressions of the form,

\[
\OmSum_{j=n}^m \int \phi_j(s,z)\,ds\bullet z
\]

These are taken to mean,

\[
\int \phi_n(s,z)\,ds\bullet \int \phi_{n+1}(s,z)\,ds\bullet ... \bullet \int \phi_m(s,z)\,ds\bullet z
\]

This is to entice the reader as to thinking that $\OmSum \int$ is its own type of operator acting on the differential form $\phi_j(s,z)\,ds\bullet z$. There are a few variations of this theme which will be used throughout this paper. But the author will attempt to maintain as much clarity as possible.

Fiddling with these objects will be the central focus of this work. For a more detailed introduction to the differential bullet product we refer to \cite{Nix}. There it is put with greater contrast to the usual integral and the development of First Order Differential Equations. It is also motivated much more aggressively.

\chapter{The Basics Of Contour Integration}\label{chap2}

\section{Introduction}\label{secC1}
\setcounter{equation}{0}

This paper is intended to set in stone the behaviour of the compositional integral in the complex plane. We will spend a large portion of time developing the intuition necessary to understand the behaviour of Compositional Contours. We will then prove multiple results about these strange contour-\textit{like} integrals.

At the present moment, we lay at a similar point Augustin-Louis Cauchy must have laid at. If $f(x,t): \mathcal{I} \times \mathbb{R} \to \mathbb{R}$ is a nice real-valued function, for an interval $\mathcal{I}$, we have a real-valued Compositional Integral; namely the function,

\[
Y_{ba}(t) = \int_a^b f(x,t)\,dx\bullet t
\]

This integral isn't necessarily defined for $a,b \in \mathcal{I}$ and $t \in \mathbb{R}$; but we know when it is and when it isn't. The function $Y_{ba}(t)$ is a nice function: typically continuous in all variables depending on lipschitz conditions. We also have The Riemann Composition of this integral, which is given as follows. Let $P = \{x_j\}_{j=0}^n$ be a partition of $[a,b]$ in \textit{descending} order, and $x_{j+1} \le x_j^* \le x_j$. Denoting $\Delta x_j = x_j - x_{j+1}$, then:

\[
Y_{ba}(t) = \lim_{\Delta x_j \to 0} \OmSum_{j=0}^{n-1} t + f(x_j^*,t)\Delta x_j \bullet t
\]

This thing converges in a sure enough manner--per the age-old Euler's method. 
So, although we have not necessarily proven this fact yet (we don't really need to, but for illustration purposes); we will provide a proof of something else which suffices for our present purposes. In that, we will prove a differing result for holomorphic functions--and the above statement is never used. Nonetheless, the proof we provide can easily be adapted to the case $f$ is Lipschitz in $t$; and the result is almost surely necessary--at least locally. Just as well this result is common knowledge; though it is more familiarly known as Euler's Method when letting the step-size approach zero. We simply choose to write it in the language of partitions.

This definition works elaborately well on the real-line--I mean it got Euler pretty far. But as Cauchy looked at the real-valued integral and wanted a complex-valued integral--we look at this. We want to add the language of arcs and contours in the complex plane to these expressions.

So to begin, we change our domain of interest. Let $\phi(s,z) : \mathcal{S} \times \mathcal{G} \to \mathcal{G}$ where $\mathcal{S}$ and $\mathcal{G}$ are domains in $\mathbb{C}$. Let's assume throughout that $\phi$ is holomorphic in both variables. If $\gamma: [a,b] \to \mathcal{S}$ is a differentiable arc in $\mathcal{S}$, then the integral along this arc is written,

\[
Y_\gamma(z) = \int_\gamma \phi(s, z)\,ds \bullet z = \int_a^b \phi(\gamma(x),z)\gamma'(x)\,dx \bullet z
\]

Where here $Y_{\gamma}(z)$ is a holomorphic function which doesn't necessarily take $\mathcal{G} \to \mathcal{G}$, but takes some subset $\mathcal{U}$ of $\mathcal{G}$ to $\mathcal{G}$. These integrals are independent of our choice of parametrization, due to the substitution law of the compositional integral. (This will also be proved in the coming section.)

This can be written a bit more conveniently as the expression,

\[
\int_\gamma \phi(s,z) \, ds \bullet z = \int_a^b \phi(\gamma(x),z)\,d\gamma \bullet z
\]

And from this, an identification in the likes of The Riemann-Stieltjes Integral can be made. We take this as the definition of our contour integral.

\[
\int_\gamma \phi(s,z) \, ds \bullet z = \lim_{\Delta x_j \to 0} \OmSum_{j=0}^{n-1} z + \phi(\gamma(x_j^*), z) \Delta \gamma_j \bullet z
\]

Where here $\{x_j\}_{j=0}^{n}$ is a partition of $[a,b]$ in descending order, and $x_{j+1} \le x_j^* \le x_j$. Further the expression $\Delta \gamma_j = \gamma(x_j) - \gamma(x_{j+1})$. This expression may be meaningless if one of the composites leaves $\mathcal{G}$. So for the moment think of this locally (small $\gamma$, $z$ in a neighborhood).

Another important consideration to make is throughout this paper we will assume $\gamma$ is continuously differentiable. We will not allow for piecewise arcs. This is done to save space; shorten proofs and discussion. However, throughout $\gamma$ could be piecewise--it would only require adding a few lines to each of the proofs.\\

There exists an algebra of arcs at our hands; but it is wildly different than the algebra which existed for Cauchy. Namely, it is non-abelian. If we have two arcs, $\gamma_1$ and $\gamma_2$, then our operation will be concatenation in some sense, and composition in another sense. We denote this $\gamma_1 \bullet \gamma_2$ which is in general non-commutative.

This operation can be surmised by the relation,

\[
Y_{\gamma_1 \bullet \gamma_2}(z) = Y_{\gamma_1}(Y_{\gamma_2}(z))
\]

This can be written more self-contained in the expression,

\[
\int_{\gamma_1 \bullet \gamma_2} \phi(s,z)\,ds \bullet z = 
\int_{\gamma_1} \phi(s,z)\,ds \bullet \int_{\gamma_2} \phi(s,z)\,ds \bullet z
\]

Where in particular $\gamma^{-1}$ is the arc which traverses backwards, similarly to Cauchy. Except the notational convenience $\gamma^{-1} = - \gamma$ is incorrect as these operations are non-abelian. Instead we are given the relation,

\[
Y_{\gamma^{-1}}(z) = Y_{\gamma}^{-1}(z) = \int_a^b -\phi(\gamma(b+a-x),z)\gamma'(b+a-x)\,dx \bullet z
\]

We can therefore think of the mapping $\gamma^{-1}:[a,b] \to \mathcal{S}$ as the arc $\gamma(b+a-x)$. We cannot make a rigorous statement of this fact without considering the domain in which $z$ is defined. As to this, the equivalence should be interpreted implicitly in neighborhoods of $z$; or as a good heuristic of what it should look like. But it will take a lot to make this correct.

This is in no way the general truth without additional information. Nonetheless, we can think of inverting a contour integral compositionally, as reversing the orientation of the contour. This can be done rigorously locally in $z$; because, as we shall see $\frac{d}{dz}Y_\gamma(z) \neq 0$. So a local holomorphic functional inverse in $z$ always exists provided we are in the co-domain of $Y_\gamma(z)$.

Now, a little nugget of gold in this notation resides. These integrals, as expected, return to the usual Cauchy kind when $\phi(s,z)$ is constant in $z$ with an added term $z$. That is to mean,

\[
\int_\gamma \phi(s)\,d s \bullet z = z + \int_\gamma \phi(s)\,d s
\]

Of which, the algebra reduces to the usual commutative algebra Cauchy envisioned. From this one can see our construction as a strict generalization of Cauchy's construction.\\

To express what we are going to do in this paper is fairly difficult. These objects are very foreign, and the symbology is novel. Of this, the reader is expected to read with care, as the author shan't pull punches. Although we mostly play a formalist's game--the $\epsilon$-$\delta$ of this work is fairly high brow.

The first thing to do is set in stone the convergence of these objects. Although we have just written a bunch of equations down, and they seem to be fairly intuitive, we do not know if we can put a stamp of $\epsilon-\delta$ approval next to them. Do these things even converge?

The second thing to do, is to prove the equivalent of Cauchy's Integral Theorem. Namely, that if $\gamma$ is a closed contour in $\mathcal{S}$ (and additionally, $\mathcal{S}$ is simply-connected) then,

\[
Y_\gamma(z) = \int_\gamma \phi(s , z) \,d s \bullet z = z
\]

This will mark the beginning of our foray. And to be flat; you could definitely prove this with a second year knowledge of ODEs. But the novelty is in the formalization.

From this we broach the concept of extending Cauchy's idea of a residue; which arises when $\phi$ has poles in the domain $\mathcal{S}$. And from this we generalize the concept of Taylor Series. As an ellipsis, in the end we introduce The Compositional Fourier Transform. So, without further ado...

\section{Normality Theorems of Contour Integrals}\label{secC2}
\setcounter{equation}{0}

In this section the following schema is used: The sets $\mathcal{S}$ and $\mathcal{G}$ are domains in $\mathbb{C}$. The function $\phi(s,z):\mathcal{S}\times \mathcal{G} \to \mathcal{G}$ is a holomorphic function in both variables. We consider a continuously differentiable arc $\gamma:[a,b]\to \mathcal{S}$. 

The aim of this section is to show The Riemann Composition converges; at least in a sufficient instance. That is to mean; let $\{x_j\}_{j=0}^{n}$ be a partition of $[a,b]$ in descending order, and $x_{j+1} \le x_j^* \le x_j$. Remembering $\Delta x_j = x_j - x_{j+1}$ and $\Delta \gamma_j = \gamma(x_j) - \gamma(x_{j+1})$, then the goal is to show,

\[
Y_{\gamma}(z) = \lim_{\Delta x_j \to 0} \OmSum_{j=0}^{n-1} z + \phi(\gamma(x_j^*),z)\Delta \gamma_j \bullet z
\]

Converges uniformly as $\Delta x_j \to 0$ on some compact subset of $\mathcal{G}$. This limit is independent of how we partition $[a,b]$; and gives a unique value for $\int_\gamma \phi(s,z)\,ds\bullet z$. In other words, we want to show that $Y_{\gamma}(z)$ is a holomorphic function taking $\mathcal{U} \to \mathcal{G}$ for some $\mathcal{U} \subset \mathcal{G}$, and it is uniquely defined.

Due to the group structure of our algebra of contours, to prove convergence we can deconstruct $\gamma$ as a collection of contours $\{\gamma_j\}_{j=1}^n$ such that $\gamma = \gamma_1 \bullet \gamma_2 \bullet ... \bullet \gamma_n = \OmSum_{j=1}^n \gamma_j$ where the length of $\gamma_j$ is less than $\rho$ for some small $\rho$--so, if the integral converges sufficiently for each arc, so does the total. This is to mean, we only need to show convergence for small arcs, and by our algebra of contours the result is derived for most arcs; at least up to some restriction to a smaller set. This will suffice for our purposes. So, without sufficient loss of generality, we will assume $\len(\gamma) \le \rho$ for some $\rho$ to be disclosed for each compact set we show convergence on.

The second thing we can do is restrict $|z-z_0| \le \delta$, so just as well we are only worried about small neighborhoods in $z$. The method requires we prove local uniform convergence. Again, this can be assumed without sufficient loss of generality.

The essential intuition is not difficult to suss out. When we worry about both variables only in a local sense; the partial compositions,

\[
\OmSum_{j=0}^{n-1} z + \phi(\gamma(x_j^*),z)\Delta \gamma_j \bullet z \approx z + \sum_{j=0}^{n-1}\phi(\gamma(x_j^*),z)\Delta \gamma_j
\]

And as we let the limit $\Delta x_j \to 0$; since the right hand side converges, we can use this to show the left hand side does as well. Though not without some massaging.

To begin we fix a compact disk $\{z \in \mathbb{C}\,:\,|z-z_0| \le P \} = \mathcal{K} \subset \mathcal{G}$ such that $|z-z_0| \le \delta < P$ lives within $\mathcal{K}$ for some $\delta>0$. We are interested in the quantity,

\[
\int_\gamma ||\phi(s,z)||_{z\in \mathcal{K}} \,|ds| \le \rho \cdot ||\phi(s,z)||_{s \in \gamma, z\in \mathcal{K}} = \kappa
\]

The value $\len(\gamma) = \rho$ can be chosen as small as we like so that $\kappa$ is as small as we like. Make a choice of $\rho$ such that $|z-z_0| \le \delta + \kappa < P$ lives inside of $\mathcal{K}$. Define $Y_n(z)$ to be the partial compositions of our integral,

\[
Y_n(z) = \OmSum_{j=0}^{n-1} z + \phi(\gamma(x_j^*),z)\Delta \gamma_j \bullet z
\]

Then the first aim is to show that for all $n$,

\[
||Y_n(z) - z||_{|z-z_0|\le\delta} \le \kappa
\]

This will give us the convenient knowledge that $Y_n$ is a normal family in the neighborhood $|z-z_0|\le \delta$. Ipso facto, this inequality is satisfied for all partitions. This can be phrased: the set of all partial compositions of the contour integral $\int_\gamma \phi(s,z)\,ds \bullet z$ form a normal family. This will be our first lemma.

\begin{lemma}\label{lma2A}
The family of functions $\mathcal{F}$ of partial compositions of the contour integral $\int_\gamma \phi\,ds\bullet z$, for each element $Y$ satisfy,

\[
||Y(z) - z||_{|z-z_0|\le\delta} \le \kappa
\]
\end{lemma}

\begin{proof}
We will prove this result by induction. But in doing so we must be very clear about what we will prove by induction. For all differentiable arcs $\gamma^*:[a^*,b^*] \to \mathcal{S}$ such that $\len(\gamma^*) = \rho^* \le \rho$ and $\gamma^* \subseteq \gamma$, and for all $|z-z_0| \le \delta$, and for all partitions $\{x_j\}_{j=0}^n$ of $[a^*,b^*]$,

\[
|\OmSum_{j=0}^{n-1} z + \phi(\gamma^*(x_j^*),z)\Delta \gamma_j^* \bullet z -z| \le \rho^* ||\phi(s,z)||_{s \in \gamma, z \in \mathcal{K}}
\]

In order to prove this result we go by induction on $n$, the length of the partition. When $n=1$ we are proving,

\[
||\phi(\gamma^*(x_0^*),z)(\gamma^*(b^*) - \gamma^*(a^*))||_{|z-z_0| \le \delta} \le \rho^* \cdot ||\phi(s,z)||_{s \in \gamma, z\in \mathcal{K}}
\] 

Since the shortest distance between two points is a straight line, we must have $|\gamma^*(b^*)-\gamma^*(a^*)|\le \rho^*$. The supremum norm handles the rest. This takes care of the case $n = 1$. Assume the case for $n$ and work on the case $n+1$.

The expression below describes the entire proof,

\begin{eqnarray*}
\OmSum_{j=0}^{n}z + \phi(\gamma^*(x_j^*),z)\Delta \gamma^*_j \bullet z = \OmSum_{j=1}^{n}z + \phi(\gamma^*(x_j^*),z)\Delta \gamma^*_j\bullet z\\ + \phi(\gamma^*(x_0^*),\OmSum_{j=1}^{n}z + \phi(\gamma^*(x_j^*),z)\Delta \gamma^*_j\bullet z)(\gamma^*(b^*) - \gamma^*(x_1^*))\\
\end{eqnarray*}

Using the induction hypothesis,

\[
|\OmSum_{j=1}^{n}z + \phi(\gamma^*(x_j^*),z)\Delta \gamma^*_j \bullet z - z| \le \rho_{-}^* \cdot ||\phi(s,z)||_{s \in \gamma, z\in \mathcal{K}}
\]

Where $\rho_{-}^{*}$ is the arc length from $\gamma^*(x_1^*)$ to $\gamma^*(a^*)$, and which since there are $n$ terms in the partition, the inequality is satisfied. Call $\rho_{+}^*$ the length of the arc from $\gamma^*(b^*)$ to $\gamma^*(x_1^*)$, then $\rho_{-}^* + \rho_{+}^* = \rho^*$. 

Lastly, the term $\OmSum_{j=1}^{n}z + \phi(\gamma^*(x_j^*),z)\Delta \gamma^*_j\bullet z$ lives in the compact set $\mathcal{K}$ because the set $|z-z_0| \le \delta + \kappa$ resides within $\mathcal{K}$ by construction. Therefore by the same argument as before, since,

\[
||\phi(\gamma^*(x_0^*),z)(\gamma^*(b^*) - \gamma^*(x_1^*))||_{|z-z_0| \le \delta + \kappa} \le \rho_{+}^*||\phi(s,z)||_{s \in \gamma, z\in \mathcal{K}}
\]

By the triangle inequality the result is proven and for all $n \in \mathbb{N}$,

\[
|\OmSum_{j=0}^{n-1} z + \phi(\gamma^*(x_j^*),z)\Delta \gamma_j^* \bullet z -z| \le \rho^* ||\phi(s,z)||_{s \in \gamma, z \in \mathcal{K}}
\]

Which gives the result,

\[
||Y(z) - z||_{|z-z_0|\le\delta} \le \kappa = \rho ||\phi(s,z)||_{s \in \gamma, z \in \mathcal{K}}
\]

\end{proof}

With this result we are halfway there. In fact, by normality, we must have sequences in $\mathcal{F}$ which converge uniformly on $|z-z_0|\le\delta$. This nearly gives the result. It tells us for some partitions of $[a,b]$ this expression converges uniformly, and $Y_\gamma(z)$ is a holomorphic function when $|z-z_0| \le \delta$; though not necessarily unique. Different limits of partitions may approach different functions.

This provides local uniform convergence for specific choices of partitions. And further, since $\kappa$ can be made arbitrarily small by choosing arbitrarily small arcs, we must have $Y_\gamma : \mathcal{K} \to \mathcal{G}$ for every compact set $\mathcal{K}$, where the length of $\gamma$ depends on $\mathcal{K}$. By the product decomposition $\gamma = \gamma_1 \bullet \gamma_2 \bullet ... \bullet \gamma_n$, we would like to have $Y_\gamma : \mathcal{G} \to \mathcal{G}$ for arbitrary arcs. This proves fairly impossible.

To give a more intuitive picture of what's going on; let $y(x) : [a,b] \to \mathbb{C}$ be the unique solution to the equation,

\[
y(x) = z + \int_a^x \phi(\gamma(u),y(u))\gamma'(u)\,du
\]

Then the function $Y_\gamma(z)$ in theory should be given by taking $y(b) = Y_\gamma(z)$. It is in no way obvious that this function lives in $\mathcal{G}$. But, if the arcs $\gamma$ are small enough, then since $\mathcal{G}$ is open, it is always possible for $Y_\gamma$ to reside in $\mathcal{G}$. Sadly though, as we begin to decompose $\gamma = \gamma_1 \bullet \gamma_2 \bullet ... \bullet \gamma_n$, where $\gamma$ is of arbitrary length and $\gamma_j$ is arbitrarily small, we may run into problems because we may begin to leave $\mathcal{G}$.

This can be better exemplified by the following simple case. Suppose we take $\mathcal{G}$ and $\mathcal{S}$ to be the unit disk $\mathbb{D}$. Let's also take the arc $\gamma:[0,1/2] \to \mathbb{D}$ to be the line $[0,1/2]$. Let's let $\phi(s,z) = z$. Then surely $\phi:\mathbb{D} \to \mathbb{D}$. But, by an old limit formula of Bernoulli and Euler,

\begin{eqnarray*}
Y_\gamma(z) &=& \int_\gamma z \,ds \bullet z\\
&=& \int_0^{\frac{1}{2}} z\,dx \bullet z\\
&=& \lim_{n\to\infty} \OmSum_{j=0}^{n-1} z + \frac{z}{2n}\bullet z\\
&=& z \lim_{n\to\infty} \prod_{j=0}^{n-1} \big{(}1 + \frac{1}{2n}\big{)}\\
&=& z \lim_{n\to\infty} \big{(}1 + \frac{1}{2n}\big{)}^n\\
&=& \sqrt{e} \cdot z\\
\end{eqnarray*}

It is no hard fact to deduce that $Y_\gamma$ does not take $\mathbb{D} \to \mathbb{D}$. So it is hopeless in general to get $Y_\gamma$ to take $\mathcal{G} \to \mathcal{G}$; unless we were to impose some extraneous condition. If we wanted to exploit the algebra of contours to their full potential we would need $Y_\gamma : \mathcal{G} \to \mathcal{G}$. And so, for our purposes, in order to do this we will eventually set $\mathcal{G} = \mathbb{C}$. But for the moment we prove convergence of the infinitely nested compositions for small arcs on arbitrary domains.\\

Again, to get a good look at what our infinitely nested compositions look like we would like to compare it to a sum. The following argument can be traced back to the work of John Gill; or at least the author learned it by way by him. Although he has used the argument multiple times, the author solely references its appearance in \cite{VirtInt}. The author has modified the conditions, and hence some of the subtleties; but the argument remains inspired by his. If we denote,

\[
z_{kn} = \OmSum_{j=k}^{n-1} z + \phi(\gamma(x_j^*),z)\Delta \gamma_j\,\bullet z
\]

With the identification that $z_{nn} = z$, and $z_{0n} = Y_n(z)$. Then, by expanding the nested composition,

\[
Y_n(z) = z + \sum_{k=0}^{n-1} \phi(\gamma(x_k^*),z_{(k+1)n})\Delta \gamma_k
\]

From this identity, and the fact $z_{kn}$ is normal in $n$ for all $k$ while $|z-z_0| \le \delta$, we essentially have the result. All that is required is to modify the proof that the usual line integral converges for holomorphic functions. This is a bit subtle, but not too difficult.

If we think of $\tau_{kn}(z) = \phi(\gamma(x_k^*),z_{(k+1)n})$ as a sample point, and rewrite this expression as,

\[
Y_n(z) = z + \sum_{k=0}^{n-1} \tau_{kn}(z) \Delta \gamma_k
\]

The result appears a lot more obvious. One can filter out all the partitions in which $|z_{(k+1)n} - z_{kn}| < \frac{M}{n}$ for all $k$ for large enough $n$ for some $M > 0$. From this, $|\tau_{(k+1)n}(z) - \tau_{kn}(z)| < \frac{M'}{n}$ for all $|z-z_0| \le \delta$ and some $M' > 0$. We can now think of this as an integral of a function of bounded variation. Or, as an implicit definition of a curve and an integral over a curve in $\mathbb{C}$. The author suggests following John B. Conway and the proof of the existence of the complex integral as presented in \cite{Con}.

Convergence must be uniform for $|z-z_0|\le\delta$ thanks to normality. The result is shown by noticing $\gamma'$ is integrable over $[a,b]$. With that, we have shown the main result of this section.

\begin{theorem}[The Convergence Theorem]\label{thmC2A}
Let $\mathcal{S}$ be a domain in $\mathbb{C}$, and let $\mathcal{G}$ be a domain in $\mathbb{C}$. Suppose $\phi(s,z):\mathcal{S} \times \mathcal{G} \to \mathcal{G}$ is a holomorphic function. For every compact set $\mathcal{K} \subset \mathcal{G}$ and every compact set $\mathcal{B} \subset \mathcal{S}$, there exists a $\rho>0$, such for all continuously differentiable arcs $\gamma:[a,b] \to \mathcal{B}$ with $\len(\gamma) \le \rho$, the contour integral,

\[
Y_\gamma(z)=\int_\gamma \phi(s,z)\,ds \bullet z = \int_a^b \phi(\gamma(x),z)\,d\gamma \bullet z
\]

Converges uniformly for $z \in \mathcal{K}$ and $\gamma \subset \mathcal{B}$, to a function $Y_\gamma : \mathcal{K} \to \mathcal{G}$.
\end{theorem}

The author would again like to thank John Gill for this part of the argument. John worked under slightly different considerations, and phrased the argument differently, but ultimately proved the same thing. He worked with arbitrary functions in $\mathbb{C} \to \mathbb{C}$, not necessarily holomorphic. He also worked with a less general construction of partitions, and arcs--and the differential relationship was phrased differently. Nonetheless he proved his result in essentially the same manner. The only piece of real novelty in our proof was the normality condition, which allowed for local uniform convergence--as is required to preserve holomorphy.

We can confidently say now, when defining,

\[
\int_{a}^w \phi(s,z)\,ds\bullet z = g(w,z)\\
\]

That $g(w,z)$ is holomorphic in both variables, so long as $z$ is restricted to a neighborhood, and $|w-a| < \rho$ is small enough. Which, this result is intended to imply that, locally, everything is holomorphic.\\

From this point out, we will assume that $\phi(s,z):\mathcal{S} \times \mathbb{C} \to \mathbb{C}$. The function $\phi$ will be entire in $z$ for the remainder of this paper. To illustrate what happens from this point it helps to place a series of somewhat anomalous examples. So to that end, let $\gamma: [a,b]\to\mathcal{S}$ be an arc and consider,

\[
Y_\gamma(z) = \int_\gamma z^2 \,ds\bullet z = \int_a^b z^2 \gamma'(x)\,dx\bullet z\\
\]

Now, unpacking this equation we are given the closed form expression for $Y_\gamma$,

\[
Y_\gamma(z) = \frac{1}{\displaystyle \frac{1}{z} + \gamma(a) - \gamma(b)}\\
\]

Which follows by differentiating by $b$, and noticing this expression satisfies,

\[
\frac{d}{db} Y_\gamma(z) = Y_\gamma(z)^2 \gamma'(b)\\
\]

And the initial condition,

\[
Y_\gamma(z)\Big{|}_{b = a} = z\\
\]

By the uniqueness of first order differential equations, these expressions are equivalent. Though we have not shown this differential relationship yet; bear with the author as our gut should say something of this nature happens. From this reduction though: if $\frac{1}{z} = \gamma(b) - \gamma(a)$ then we get a pole in $Y_\gamma$. So, it has poles in $z$ depending on $\gamma$. What we want to say is that this is \emph{sort of} the worst we can get when our integrand is holomorphic.  This proves to corrupt what would otherwise be a simple theory.

It helps to add even more anomalous examples. And the second kind of anomaly is more indicative of the true nature of these things. Take the function $\phi(s,z) = z^3$. Then,

\[
Y_\gamma = \int_\gamma z^3\,ds\bullet z = \frac{1}{\displaystyle\sqrt{\frac{1}{z^{2}} + 2\gamma(a) - 2\gamma(b)}}\\
\]

Because,

\[
\frac{d}{db}Y_\gamma = Y_\gamma^3 \gamma'(b)\\
\]

And the initial condition $Y_\gamma \Big{|}_{b=a} = z$. This doesn't look promising. We have not just poles, but actual branch cuts. There's nothing worse to a complex analyst than branch-cuts. It doesn't end there. Just as anomalous,

\[
Y_\gamma = \int_\gamma e^{-z}\,ds\bullet z = \log(e^z + \gamma(b)-\gamma(a))\\
\]

Because,

\[
\frac{d}{db}Y_\gamma = e^{\displaystyle -Y_\gamma} \gamma'(b)\\
\]

With the same initial condition $Y_\gamma \Big{|}_{b=a} = z$. Here is where we start to see the shape of these things. Obviously $Y_\gamma$ in general cannot be meromorphic. It cannot converge everywhere to a meromorphic function because intrinsically the above functions have branches. But we can't really talk about branches. Insofar as $Y_\gamma(z)$ is a point-wise defined function.

If we were to rewrite this $Y_\gamma$ we can come a little closer to a more manageable form.

\[
Y_\gamma(z) = z + \log \big{(}1+\frac{\gamma(b) - \gamma(a)}{e^z}\big{)}\\
\]

Wherein this form of the equation satisfies $Y_\gamma \Big{|}_{b=a} = z$, without the pesky need for a lot of branch-cut talk. However, we are still guaranteed that this thing doesn't converge when $z = \log(\gamma(a) - \gamma(b)) + 2 \pi i k$ for $k\in \mathbb{Z}$, and cannot in any neighborhood of these points.

All that can be said is that there are definitely points of non-convergence in $Y_\gamma$. This seems to provide a contradiction because for every neighborhood in $z$ we can define contour integrals $Y_\gamma$ where $\len(\gamma) \le \rho$. And we can compose this with other contour integrals in this neighborhood $\gamma = \gamma_1 \bullet \gamma_2\bullet...\bullet \gamma_n$; to make $\len(\gamma)$ of arbitrary length. Naturally because $\phi(s,z)$ is entire in $z$ we only hit problems when we hit an infinite value. This proves to be a very deep foundational issue.

The author, despite his best attempts to clear the air, still gets confused about this, but he is sure of one thing. We can always find a neighborhood in $z$ where $Y_\gamma$ converges for arbitrary $\gamma$. In fact, we can cover the plane with such $z$ up to a set of measure zero in $\mathbb{C}$. And this is how we'll think about it. There is an open set $\mathcal{D} \subset \mathbb{C}$ (not necessarily connected) such that $Y_\gamma: \mathcal{D} \to \mathbb{C}$. This should be fairly intuitive; the only requirement of difficulty is that $\overline{\mathcal{D}} = \mathbb{C}$. So, we have convergence almost everywhere.

To those familiar with the eccentricities of Complex Dynamics; we can think about $\mathcal{D}$ as the Fatou set. It is the set in which our partial compositions are normal; where the infinitely nested compositions are normal. Then $\mathbb{C}/\mathcal{D}$ is a measure zero Julia set; some irregular (probably fractal) measure zero set in $\mathbb{C} = \mathbb{R}^2$; where our partial compositions are \emph{not} normal. By way of this, we split the plane into two sets $\mathcal{D}$ and $\mathbb{C}/\mathcal{D}$; where $\mathcal{D}$ is a domain of normality, the other isn't. And just as necessarily $\mathcal{D} = \mathcal{D}(\gamma)$. The set $\mathcal{D}$ varies as we let $\gamma$ vary.

We will not need these details very much. The author is not leveraging this entire paper on this analysis; however it's a good way to paint a picture of what these things begin to look like. The better way, and the correct way, is to think of $Y_\gamma$ as a \emph{sheaf} defined locally in $z$. We won't be calling on things like Liouville's Theorem, Weierstrass Product Formula, or Mittag-Leffler's Theorem; or other properties of functions holomorphic/meromorphic on $\mathbb{C}$; so we can play a little fast and loose. Our discussion shall be entirely based upon local behaviour in $z$. In many instances, we will not see a clean formula, but instead equivalences between types of contour integrals. This is the point of interest; the ways these things equal each other; less-so the nature of convergence; or the where or when. The only principle we call upon gregariously is The Identity Theorem--which is certainly a local property.

For the same reasons now, if we allow $\mathcal{G} = \mathbb{C}$, so that our domain is unbounded, then The Convergence Theorem \ref{thmCVG} can be expanded into a much more complex beast. First, by noticing $\gamma = \gamma_1 \bullet \gamma_2 \bullet ... \bullet \gamma_n$ where $\gamma_j$ is arbitrarily small; and since our domains are unbounded in any direction; we will not run into any undefined compositions. However, we may encounter infinities. The sole argument we need to make, is that these domains of non-convergence are managed by the differential properties of $Y_\gamma$. So as we take small arcs and compose them together, we no longer hit a boundary where our functions are not necessarily undefined; but our differential equation may force our integral to blow up to infinity; and locally our compositions look like solutions to these differential equations.

\begin{theorem}\label{thmDNSCVG}
Let $\mathcal{S}$ be a domain in $\mathbb{C}$. Suppose $\phi(s,z):\mathcal{S} \times \mathbb{C} \to \mathbb{C}$ is a holomorphic function. For all continuously differentiable arcs $\gamma:[a,b] \to \mathcal{S}$, the contour integral,

\[
Y_\gamma(z)=\int_\gamma \phi(s,z)\,ds \bullet z = \int_a^b \phi(\gamma(x),z)\,d\gamma \bullet z
\]

Converges uniformly on compact subsets of $\mathcal{D} = \mathcal{D}(\gamma)\subset\mathbb{C}$, to a holomorphic function $Y_\gamma : \mathcal{D} \to \mathbb{C}$. Such that $\overline{\mathcal{D}} = \mathbb{C}$.
\end{theorem}

\begin{proof}
The proof of this fact will be belayed slightly. Init, what we especially need is that,

\[
\frac{d}{d b} Y_\gamma(z) = \phi(\gamma(b),Y_\gamma(z))\gamma'(b)\\
\]

With the initial condition $Y_\gamma(z)\big{|}_{b=a} = z$. From here, all that's needed is that differential equations with holomorphic integrands can at worst have a set of discontinuities $z \in \mathcal{Y}$ such that $\mathcal{Y}$ is of measure-zero in $\mathbb{C}$. This is standard work. See The Differential Theorem \ref{thmDIFF} below for a complete proof of the above equation. Both theorems are then proved in tandem; if the reader care to observe.
\end{proof}

The reader should note now, directly, that the author has absolutely no clue what $\mathcal{D}$ is. He only knows it's there. But it is irrelevant to the discussion to follow. All we need to know from this is that $Y_{\gamma_1}(z)$ is holomorphic on $\mathbb{C}$ minus a set of measure-zero; and so is $Y_{\gamma_2}$; and therefore $Y_{\gamma_1} \bullet Y_{\gamma_2}$ is also holomorphic on $\mathbb{C}$ minus a set of measure-zero. And $Y_{\gamma_1} \bullet Y_{\gamma_2} = Y_{\gamma_1\bullet \gamma_2}$. So we can now think of equivalence almost everywhere when making compositions.\\

The second way of viewing convergence, and the more grounded way, is focused on fixed-points. Supposing additionally $\phi(s,z_0) = 0$ for some $z_0 \in \mathbb{C}$ for all $s \in \mathcal{S}$. Then of this course, looking at the partial compositions, ${\displaystyle \int_\gamma \phi(s,z)\,ds\bullet z \Big{|}_{z= z_0}} = z_0$. Consequently, we can always find a neighborhood of $z_0$ such that ${\displaystyle \int_\gamma \phi(s,z)\,ds\bullet z}$ is holomorphic.

To ratify the discussion; if $\phi(s,z_0) = 0$, then there exists a neighborhood $\mathcal{U} = \{|z-z_0|<\delta\}$ such $\int_\gamma \phi(s,z)\,ds\bullet z$ is holomorphic on $\mathcal{U}$ for $\len(\gamma) \le \rho$ as small as possible. Then compositions $\gamma = \gamma_1 \bullet \gamma_2 \bullet \cdots \bullet \gamma_n$ exist as long as we shrink $\mathcal{U}$ to as small as necessary so the compositions are valid. This is to say that a domain of normality exists around $z_0$ where $\phi(s,z_0) = 0$ for all $\gamma$. In this language we can write,

\begin{theorem}\label{thmLOCAL}
Let $\mathcal{S}$ be a domain in $\mathbb{C}$. Suppose $\phi(s,z):\mathcal{S} \times \mathbb{C} \to \mathbb{C}$ is a holomorphic function. Additionally, assume that $\phi(s,z_0) = 0$ for some $z_0 \in \mathbb{C}$. For a continuously differentiable arc $\gamma:[a,b] \to \mathcal{S}$, the contour integral,

\[
Y_\gamma(z)=\int_\gamma \phi(s,z)\,ds \bullet z = \int_a^b \phi(\gamma(x),z)\,d\gamma \bullet z
\]

Converges uniformly in a neighborhood $|z-z_0| < \delta$ for some $\delta > 0$.
\end{theorem}

From this version of the expansion; we can see that if $Y_\gamma$ is holomorphic on $\mathcal{U}$ a neighborhood of $z_0$; and similarly with $Y_\tau$ but with a neighborhood $\mathcal{U}'$. Then, $Y_\gamma \bullet Y_\tau = Y_{\gamma \bullet \tau}$ on $Y_\tau(\mathcal{U}') \cap \mathcal{U} = \mathcal{U}'' \neq \emptyset$ (because necessarily a neighborhood of $z_0$ is contained within $\mathcal{U}''$). At this point we can see the discussions of \textit{sheafs} are well at hand.

The first order expansion of $Y_\gamma$ in this instance is rather nice too,

\[
Y_\gamma(z) = z_0 + (z-z_0)e^{\displaystyle \int_\gamma \frac{\partial \phi}{\partial z}(s,z_0)\,ds} + \mathcal{O}(z-z_0)^2\\
\]

And there exists a recursive pattern to each Taylor coefficient about $z_0$, to brute force it is a bit unnecessary for this paper; but it's possible.

The sole purpose of this brief detour is to show an explicit example of normality; so that we are not all up in the air. We can think of this as a more strict form of the normal convergence we deliberated upon above. Where now we have an actual neighborhood rather than a nondescript domain.\\

The last thing we will state in this section is something implicitly proved in the above discussions of convergence. We got a little caught up in global properties, we missed out on more local opportunities. As in the above corollary, we will assume that $\phi(s,z): \mathcal{S} \times \mathbb{C} \to \mathbb{C}$. We will attempt to obtain a bound on,

\[
Y_\gamma(z) = \int_\gamma \phi(s,z)\,ds\bullet z\\
\]

Taking the partial compositions,

\[
z_{kn} = \OmSum_{j=k}^{n-1} z + \phi(\gamma(x_j^*),z)\,\Delta \gamma_j \bullet z\\
\]

Which satisfy,

\[
\int_\gamma \phi(s,z)\,ds\bullet z = z + \lim_{n\to\infty} \sum_{j=0}^{n-1} \phi(\gamma(x_j^*),z_{(j+1)n})\Delta \gamma_j\\
\]

Now on the compact set $\mathcal{K}$, if $Y_\gamma(z)$ is continuous there, there exists a bound $M$ such that $\sup_{j,n}||z_{jn}||_{z \in \mathcal{K}} \le M$. Call $\mathcal{L}$ a compact set such that all of its elements $z$ satisfy $|z| \le M$. Then, necessarily,

\begin{eqnarray*}
\Big{|}\Big{|}\int_\gamma \phi(s,z)\,ds\bullet z - z\Big{|}\Big{|}_{z \in \mathcal{K}} &\le& \lim_{n\to\infty} \sum_{j=0}^{n-1} ||\phi(\gamma(s_j^*),z)||_{z \in \mathcal{L}}|\Delta \gamma_j|\\
&\le& \int_\gamma ||\phi(s,z)||_{z\in\mathcal{L}}\,|ds|\\
\end{eqnarray*}

This allows us to make the following claim. This is the equivalent of a triangle inequality, and for all extensive purposes works as a triangle inequality. We'll call upon this principle frequently throughout this paper, but it'll be done with little reference to the following theorem. But, the reader should be able to spot it when we use it; and the author will try to be as explicit as possible.

\begin{theorem}[The Triangle Inequality Theorem]\label{thmTriInq}
Let $\mathcal{S}$ be a domain in $\mathbb{C}$. Suppose $\phi(s,z):\mathcal{S} \times \mathbb{C} \to \mathbb{C}$ is a holomorphic function. Let $\gamma:[a,b] \to \mathcal{S}$ be a continuously differentiable arc. If $\displaystyle{\int_\gamma \phi(s,z)\,ds\bullet z}$ is continuous on the compact disk $\mathcal{K} \subset \mathbb{C}$; there exists a larger compact disk $\mathcal{L} \subset \mathbb{C}$ such that $\mathcal{K} \subset \mathcal{L}$ and,

\[
\Big{|}\Big{|}\int_\gamma \phi(s,z)\,ds\bullet z - z\Big{|}\Big{|}_{z \in \mathcal{K}} \le \int_\gamma ||\phi(s,z)||_{z\in\mathcal{L}}\,|ds|\\
\]
\end{theorem}

\section{An extension of Cauchy's Integral Theorem}\label{secC3}
\setcounter{equation}{0}

In this section the following schema is used: $\mathcal{S}$ is a simply connected domain in $\mathbb{C}$ and $\phi(s,z) : \mathcal{S} \times \mathbb{C} \to \mathbb{C}$ is a holomorphic function. 

Simple connectivity is necessary in this section; everything will have to do with closed contours. A closed contour $\gamma$ will always be continuously differentiable--for the sake of transparency. We can begin by assuming that $\len(\gamma) \le \rho$ for small $\rho$; then upon completing this section, we can turn to Theorem \ref{thmDNSCVG}; and once this is justified; return to this section and repeat without the restriction $\len(\gamma) \le \rho$.

There are two things to prove in this section. But they are the same thing phrased in different manners. Define the function,

\[
y(x) = \int_a^x \phi(\gamma(u),z) d\gamma \bullet z
\]

Then in theory, this function satisfies the differential equation $y(a) = z$ and,

\[
y'(x) = \phi(\gamma(x),y(x))\gamma'(x)
\]

Which aligns with the case for the real-valued compositional integral. This written more explicitly becomes,

\[
y(x) = z + \int_a^x \phi(\gamma(u),y(u))\gamma'(u)\,du
\]

Now if we were to alter this to a contour notation, where $\gamma_ w:[a,x] \to \mathcal{S}$, $ w = \gamma(x)$ and $ w_0 = \gamma(a)$; this becomes the expression,

\[
g( w) = \int_{\gamma_ w} \phi(s,z)\,ds \bullet z
\]

Then the idea of this section is to prove the function $g( w)$ does not depend on the path $\gamma_ w$, but only on the end points $w$ and $ w_0$. That is to mean, it doesn't matter how we define this path; or what route we take. So long as it ends at $ w \in \mathcal{S}$ and begins at $ w_0 \in \mathcal{S}$, the resultant is the same. We find this rather naturally in Cauchy's analysis, but it's phrased with the use of a constant of integration. This is also to mean,

\[
\int_\gamma \phi(s,z) \,ds \bullet z = \int_\tau \phi(s,z)\,ds\bullet z
\]

If $\tau$ and $\gamma$ have the same endpoints and share the same orientation. This is certainly true for the example $\phi(s,z) = z^2$. Wherein,

\[
\int_\gamma z^2\,ds\bullet z = \frac{1}{\displaystyle \frac{1}{z} +  w_0 -  w}\\
\]

Which is distinctly only dependent on the end-points of the contour. This is so, regardless of the fact there is a pole somewhere in $z$.\\

The second, and more precise manner of phrasing this result is that we have something like Cauchy's Integral Theorem. This is to mean, if $\gamma$ is a closed contour in $\mathcal{S}$--$ w =  w_0$, then,

\[
\int_\gamma \phi(s,z)\,ds \bullet z = z
\]

So, as we have an algebra of contours across addition as Cauchy defined it, where a closed contour was essentially the identity (the value $0$); we have something similar here. Our algebra of contours is across composition, and is non-abelian; but a closed contour is still the identity. Namely, it is the function $z \mapsto z$. Or at least, once we evaluate the contour integral on some holomorphic function. This will set the stage for a lot of algebraic analysis.

To derive the first result from the second, note that $\gamma$ can always be decomposed as two contours $\gamma_1 \bullet \gamma_2$, and since,

\[
z = Y_\gamma(z) = Y_{\gamma_1 \bullet \gamma_2}(z) = Y_{\gamma_1}(Y_{\gamma_2})
\]

We must have,

\[
Y_{\gamma_2}^{-1} = Y_{\gamma_1}
\]

Where $\gamma_2$ shares the same endpoints as $\gamma_1$, but has opposite orientation; hence the functional inverse. Much care must be taken in making this statement though, as these are not necessarily biholomorphic functions of $\mathbb{C}$. An inverse does not necessarily exist. In this sense, we are meant to interpret these equations implicitly in neighborhoods of $z$. This equation is to be taken loosely.

To accent, this result is not very hard. When we combine these two intuitions the result is in front of us. If $\gamma$ is a closed contour then,

\[
Y_\gamma(z) = z + \int_\gamma \phi(s,g(s))\,ds
\]

And since $g'( w) = \phi( w,g( w))$, by Cauchy's Integral Theorem it must follow $Y_\gamma(z) = z$. Any function $h$ with an antiderivative satisfies $\int_\gamma h(s)\,ds = 0$. The key to all of this then relies in showing,

\[
g( w) = \int_{\gamma_ w} \phi(s,z)\,ds \bullet z
\]

satisfies,

\[
\frac{d}{d w} g( w) = \phi( w,g( w))
\]

With this all our intuition makes sense and becomes clear. We make these statements precise.

\begin{theorem}[The Differential Theorem]\label{thmDIFF}
Suppose $\mathcal{S}$ is a simply connected domain in $\mathbb{C}$. Let $\phi(s,z) : \mathcal{S} \times \mathbb{C} \to \mathbb{C}$ be a holomorphic function. Suppose $\gamma_ w:[a,x] \to \mathcal{S}$ is a continuously differentiable arc in $\mathcal{S}$. Let $ w = \gamma(x)$ and $ w_0 = \gamma(a)$. Then the function,

\[
g( w) = \int_{\gamma_ w} \phi(s,z)\,ds\bullet z
\]

Satisfies the differential equation; $g( w_0) = z$ and,

\[
\frac{d}{d w} g( w) = \phi( w,g( w))
\]
\end{theorem}

\begin{remark}
We will only provide a proof sketch of this result. The author found a more in-depth analysis in many senses a waste of space; as the result is perfectly intuitive, and filling in the gaps is more a manner of book-keeping a bunch of $\epsilon$'s and $\delta$'s than real insight. 
\end{remark}

\begin{proof}
Taking our previous notation, letting,

\[
z_{1n} = \OmSum_{j=1}^{n-1} z + \phi(\gamma(x_j^*),z)\Delta \gamma_j \bullet z
\]

And recalling,

\[
g_n( w) = \OmSum_{j=0}^{n-1} z + \phi(\gamma(x_j^*),z)\Delta \gamma_j \bullet z
\]

Then two things should be clear, $\lim_{n\to\infty} z_{1n} = \lim_{n\to\infty} g_n( w)$. Second of all, these two functions are related by the identity,

\[
g_n( w) = z_{1n} + \phi(\gamma(x_0^*),z_{1n})(\gamma(x) - \gamma(x_1))
\]

There are three things which can be observed here. Firstly $\gamma(x_0^*) \to  w$ as $n\to\infty$. Secondly, $\gamma(x) =  w$ and $\gamma(x_1)$ can be taken to be $ w'$; of which $ w -  w'$ tend to zero as $\mathcal{O}(1/n)$. Just as well we have the notational convenience $g_{n-1}( w') = z_{1n}$. Therefore, if we rearrange this expression,

\[
\frac{g_n( w) - z_{1n}}{\gamma(x) - \gamma(x_1)} = \phi(\gamma(x_0^*),z_{1n})
\]

And then rewrite it through asymptotic equivalence,

\[
\frac{g_n( w) - g_{n-1}( w')}{ w -  w'} = \phi( w, g_{n-1}( w'))
\]

Where, filling in the gaps, connecting the dots; while $|w-w'| < \delta$ and $n>N$; by normality,

\begin{eqnarray*}
|g_{n-1}(w') - g(w)| &<& \epsilon\\
|\phi(w,g_{n-1}(w')) - \phi(w,g(w))| &<& \epsilon'\\
\end{eqnarray*}

So that,

\begin{eqnarray*}
\Big{|}\frac{g_n( w) - g_{n-1}( w')}{ w -  w'} - \frac{g( w) - g( w')}{ w -  w'}\Big{|} &<& \epsilon''\\
\Big{|} \phi(w,g(w)) - \frac{g( w) - g( w')}{ w -  w'}\Big{|} &<& \epsilon'''\\
\end{eqnarray*}

We arrive at our derived result, though not without some extra elbow-grease left to the reader.

\[
\lim_{ w' \to  w} \frac{g( w) - g( w')}{ w -  w'} = \phi( w,g( w))
\]

This provides the result. But, before closing the proof; the author would like to elaborate a tad more. We should recall that $g_n(w)$ converges uniformly in $w$ and $z$. So as this may have seemed a tad simplistic; the fact we know all these functions are normal families to begin with, makes it concise rather.
\end{proof}

With this we state the final result of the section.

\begin{theorem}[The Compositional Integral Theorem]\label{thmC3A}
Let $\mathcal{S}$ be a simply connected domain in $\mathbb{C}$. Suppose $\phi(s,z):\mathcal{S} \times \mathbb{C} \to \mathbb{C}$ is a holomorphic function. For all closed contours $\gamma:[a,b] \to \mathcal{S}$, the contour integral,

\[
\int_\gamma \phi(s,z)\,ds \bullet z = z
\]
\end{theorem}

\begin{proof}
Define $g( w,z)$ as the solution to the equations,

\[
\frac{d}{d  w} g( w,z) = \phi( w,g( w,z))\\
\]

And,

\[
g(\alpha, z) = g(\gamma(a),z) = z\\
\]

Where $ w \in \gamma$ traces along the contour, starting and finishing at the point $\alpha \in \gamma$. Choose a small enough neighborhood in $z$ in a connected component of $\mathcal{D}$ (the domain of normality), so that,

\[
g( w,z) = \lim_{\Delta \gamma_j \to 0} \OmSum_{j=0}^{n-1} z + \phi(\gamma_j^*,z)\Delta \gamma_j \bullet z\\
\]

Converges and is analytic. Therefore, as $ w \to \alpha^-$ (approaches $\alpha$ along the end of the path $\gamma$) we get,

\[
g( w,z) \to \int_{\gamma} \phi(s,z)\,ds\bullet z\\
\]

However, $g$ also satisfies,

\[
g(\alpha^-,z) = z + \int_\gamma \phi(s,g(s,z))\,ds\\
\]

But, $\phi(s,g(s,z)) = \dfrac{d g}{ds}$. Which implies,

\[
\int_\gamma \phi(s,g(s,z))\,ds = \int_\gamma dg = 0\\
\]

Therefore the result for $z$ in our neighborhood. By analytic continuation it holds everywhere in the domain of normality.
\end{proof}

\section{A formula for the derivative}\label{secC4}
\setcounter{equation}{0}

In this section we push towards understanding the action of $\frac{d}{d z}$ applied to our contour integral. We'll derive a closed form expression; and consequently show that our derivative is non-vanishing in $z$. So in such a sense, our contour integral is always locally invertible. And going further, this lets us speak about the orientation of our contour with more frankness.

In this section the following schema is used. The set $\mathcal{S}$ is a domain in $\mathbb{C}$. The function $\phi(s,z) : \mathcal{S} \times \mathbb{C} \to \mathbb{C}$ is a holomorphic function. Let $\gamma : [a,b] \to \mathcal{S}$ be a continuously differentiable arc. To begin we'll look at the contour integral in question,

\[
Y_\gamma(z) = \int_{\gamma} \phi(s,z)\,ds\bullet z\\
\]

Take its partial compositions,

\[
Y_n(z) = \OmSum_{j=0}^{n-1} z + \phi(\gamma_j^*,z)\Delta \gamma_j\bullet z\\
\]

As before, let us denote,

\[
z_{kn} = \OmSum_{j=k}^{n-1} z + \phi(\gamma_j^*,z)\Delta \gamma_j\bullet z\\
\]

Take the derivative of $Y_n$ in $z$; and denote $\frac{d}{d z} \phi(s,z) = \phi'(s,z)$. Then by the chain rule,

\[
\frac{d}{d z} Y_n(z) = \prod_{j=0}^{n-1}\Big{(}1 + \phi'(\gamma_j^*,z_{(j+1)n})\Delta \gamma_j\Big{)}\\
\]

If we take $\log$ on both sides of this equation, and we make the equivalence $\log (1 + \Delta) \sim \Delta$ as $\Delta \to 0$, we are given,

\begin{eqnarray*}
\log \frac{d}{d z} Y_n(z) &=& \sum_{j=0}^{n-1} \log\big{(}1 + \phi'(\gamma_j^*,z_{(j+1)n})\Delta \gamma_j\big{)}\\
&\sim& \sum_{j=0}^{n-1} \phi'(\gamma_j^*,z_{(j+1)n})\Delta \gamma_j\\
\end{eqnarray*}

And here is where we fiddle slightly. Denote $\gamma_ w : [a, w] \to \mathcal{S}$ as the contour $\gamma$ up to $w$. So that $\gamma_b = \gamma$, and $\gamma_a$ is the null contour and $ w = \gamma(x)$. Essentially the following equality happens,

\[
\lim_{n\to\infty} \sum_{j=0}^{n-1} \phi'(\gamma_j^*,z_{(j+1)n})\Delta \gamma_j = \int_\gamma \phi'\Big{(} w,\int_{\gamma_ w}\phi(s,z)\,ds\bullet z\Big{)}\,d w\\
\]

Exponentiating this result, and collecting everything together,

\[
\frac{d}{d z} Y_\gamma(z) = e^{\displaystyle \int_\gamma \phi'\Big{(} w,\int_{\gamma_ w}\phi(s,z)\,ds\bullet z\Big{)}\,d w}\\
\]

Writing the result in this manner is the desired form. But we need not make things so complicated to prove this result. Therefore in the theorem below, it is exactly what we've just written, but we can write it a bit simpler. In doing such, we partially abandon the fact that we are using contours, and think of this solely as an integration.

\begin{theorem}[The Derivative Formula]\label{thmDEVFRM}
Let $a, b \in \mathbb{R}$ with $a \le b$. Suppose $g(t,z) : [a,b] \times \mathbb{C} \to \mathbb{C}$ is continuously differentiable in $t$ and holomorphic in $z$. Then,

\[
\frac{d}{d z}\int_a^b g(t,z) \,dt\bullet z = e^{\displaystyle \int_{a}^b g'\Big{(}x,\int_a^x g(t,z)\,dt\bullet z\Big{)}\,dx}\\
\]
\end{theorem}

\begin{proof}
All that must be shown is that the limit below converges to the desired result. Let $\{x_k\}_{k=0}^{n}$ be a partition of $[a,b]$ in descending order, let $x_{k+1} \le x_k^* \le x_k$, and let $\Delta x_k = x_k - x_{k+1}$; then what we want is,

\[
\lim_{\Delta x_k \to 0} \sum_{k=0}^{n-1} g'(x_k^*, z_{(k+1)n})\Delta x_k = \int_{a}^b g'\Big{(}x,\int_a^x g(t,z)\,dt\bullet z\Big{)}\,dx\\
\]

Where,

\[
z_{kn} = \OmSum_{j=k}^{n-1} z + g(x_j^*,z)\Delta x_j \bullet z\\
\]

But by the convergence of the nested partial compositions,

\[
z_{kn} - \int_{a}^{x_k^*} g(t,z)\,dt\bullet z \to 0
\]

And therefore,

\[
\sum_{k=0}^{n-1} g'(x_k^*, z_{(k+1)n})\Delta x_k -  g'\Big{(}x_k^*,\int_{a}^{x_k^*} g(t,z)\,dt\bullet z\Big{)}\Delta x_k \to 0
\]

But,

\[
\sum_{k=0}^{n-1}g'\Big{(}x_k^*,\int_{a}^{x_k^*} g(t,z)\,dt\bullet z\Big{)}\Delta x_k \to \int_a^b g'\Big{(}x, \int_{a}^{x} g(t,z)\,dt\bullet z\Big{)}\,dx\\
\]

Which gives the result.
\end{proof}

Since,

\[
\int_{\gamma} \phi(s,z) \,ds\bullet z = \int_a^b \phi(\gamma(x),z)\gamma'(x)\,dx \bullet z\\
\]

We derive our desired form of the equality. Now, all of these contour integrals have non-vanishing derivatives. Which is an easy consequence of our derivative being the exponential of some holomorphic function. This is to say,

\begin{theorem}\label{thmNONZERO}
Suppose $\mathcal{S}$ is a domain in $\mathbb{C}$. Suppose that $\phi(s,z) : \mathcal{S} \times \mathbb{C} \to \mathbb{C}$ is a holomorphic function. Let $\gamma : [a,b] \to \mathcal{S}$ be a differentiable arc. Then,

\[
\frac{d}{d z} \int_\gamma \phi(s,z)\,ds\bullet z \neq 0\\
\] 
\end{theorem}

And from this we are given an approach at better understanding orientation. So now, when we speak of inverting contours we can speak fairly absolutely. A local inverse of $Y_\gamma$ exists everywhere save a set of measure zero (a domain of non-normality). It is much easier now to make the claim that $Y_{\gamma^{-1}} = Y_\gamma^{-1}$. Of this we can now speak frankly of orientation, and our interjections above can be made definite.

The author chooses to write this in a single theorem. This theorem is quite frankly the best way we can talk about orientation. And furthermore, it serves to demonstrate the idea clearly: functional inversion inverts the orientation.

\begin{theorem}[The Orientation Theorem]\label{thmORNT}
Let $\mathcal{S}$ be a domain in $\mathbb{C}$. Suppose $\phi(s,z) : \mathcal{S} \times \mathbb{C} \to \mathbb{C}$ is a holomorphic function. Prescribe $\gamma:[a,b] \to \mathcal{S}$ as a differentiable arc, and $\gamma^{-1}:[a,b] \to \mathcal{S}$ as the arc $\gamma^{-1}(x) = \gamma(b+a-x)$. If,

\[
Y_\gamma(z) = \int_{\gamma}\phi(s,z)\,ds\bullet z
\]

Then about every $z_0 \in Y_\gamma(\mathcal{D})$ (for $\mathcal{D}$ the domain of normality) there is a neighborhood $\mathcal{U}$ where,

\[
\int_{\gamma^{-1}}\phi(s,z)\,ds\bullet z = Y_\gamma^{-1}(z)
\]

And $Y_\gamma^{-1}$ is the functional inverse of $Y_\gamma$ on $\mathcal{U}$.
\end{theorem}

\begin{proof}
Starting with the contour integral of interest,

\[
Y_\gamma(z) = \int_\gamma \phi(s,z)\,ds\bullet z\\
\]

Take its partial compositions,

\[
Y_n(z) = \OmSum_{j=0}^{n-1} z + \phi(\gamma_j^*,z)\Delta \gamma_j\bullet z
\]

And invert it functionally about $z_0$, which is possible because for large enough $n$, the derivative $Y_n'(z_0) \neq 0$; and each individual term has derivative in a neighborhood of $1$. Then,

\[
Y_n^{-1} = \MhSum_{j=0}^{n-1} \Big{(}z + \phi(\gamma_j^*,z)\Delta \gamma_j\Big{)}^{-1}\bullet z\\
\]

For the moment, assume for small enough $\Delta \gamma_j$, the inverse,

\[
\Big{(}z + \phi(\gamma_j^*,z)\Delta \gamma_j\Big{)}^{-1} \sim z - \phi(\gamma_j^*,z)\Delta \gamma_j\\
\]

With this assumption,

\[
\lim_{n\to\infty}Y^{-1}_n(z) = \lim_{n\to\infty}\MhSum_{j=0}^{n-1} z - \phi(\gamma_j^*,z)\Delta \gamma_j\bullet z\\
\]

But this is precisely the partial compositions of the contour integral across $\gamma^{-1}(x) = \gamma(b+a-x)$. So proving our assumption proves the result. The rest of this proof then devolves into,

\[
\dfrac{\big{(}z+ \phi(\gamma_j^*,z)\Delta \gamma_j\big{)} \bullet \big{(}z - \phi(\gamma_j^*,z)\Delta \gamma_j\big{)}\bullet z - z}{\Delta \gamma_j}\to 0\,\, \text{as}\,\, \Delta \gamma_j \to 0\\
\]

Which can be translated more strictly as; for some $A \in \mathbb{R}^+$,

\[
\Big{|}\Big{|}\big{(}z+ \phi(\gamma_j^*,z)\Delta \gamma_j\big{)} \bullet \big{(}z - \phi(\gamma_j^*,z)\Delta \gamma_j\big{)}\bullet z - z\Big{|}\Big{|}_{z\in\mathcal{K}} \le A \Delta \gamma_j^2
\]

For every compact set $\mathcal{K} \subset \mathbb{C}$. This result is seen by direct comparison. Expanding the composition, we get,

\[
\phi(\gamma_j^*,z-\phi(\gamma_j^*,z)\Delta\gamma_j)\Delta \gamma_j - \phi(\gamma_j^*,z)\Delta\gamma_j = \mathcal{O}(\Delta \gamma_j^2)\\
\]

Which speaks for itself. Therefore in our neighborhood,

\[
\big{(}z + \phi(\gamma_j^*,z)\Delta \gamma_j\big{)}^{-1} = z - \phi(\gamma_j^*,z)\Delta \gamma_j + \mathcal{O}(\Delta \gamma_j^2)\\
\]

And therefore under the operator $\OmSum$ the limits remain equivalent.
\end{proof}

And lastly, from this theorem we are guaranteed something we hinted at earlier. By The Compositional Integral Theorem \ref{thmC3A} and by the locally invertible property of the contour integral, our anti-derivatives are independent of the path. This is to mean,

\begin{theorem}[The Path Independence Theorem]\label{thmPTHIND}
Suppose that $\mathcal{S}\subset \mathbb{C}$ is a simply connected domain. Suppose that $\phi(s,z):\mathcal{S} \times \mathbb{C} \to \mathbb{C}$ is a holomorphic function. Suppose $\tau:[a,b] \to \mathcal{S}$ and $\gamma:[a,b] \to \mathcal{S}$ are differentiable arcs such that: $\tau(a) = \alpha = \gamma(a)$, and $\tau(b) = \beta = \gamma(b)$. Then,

\[
\int_\gamma \phi(s,z)\,ds\bullet z = \int_\tau \phi(s,z)\,ds\bullet z\\
\]
\end{theorem}

\begin{proof}
Pick $z_0 \in \mathbb{C}$. The arc $\gamma \bullet \tau^{-1}$ is a closed contour, and for $z$ in a neighborhood $\mathcal{U}$ of $z_0$; by The Compositional Integral Theorem \ref{thmC3A}:

\[
\int_{\gamma \bullet \tau^{-1}}\phi(s,z)\,ds\bullet z = z\\
\]

By The Orientation Theorem \ref{thmORNT}, the neighborhood $\mathcal{U}$ can be chosen as small as possible so that ${\displaystyle \int_{\tau^{-1}} = \Big{(}\int_\tau\Big{)}^{-1}}$ on $\mathcal{U}$. Therefore on $\mathcal{U}$,

\[
\int_\gamma \phi(s,z)\,ds\bullet z = \int_\tau \phi(s,z)\,ds\bullet z\\
\]

By analytic continuation it must hold everywhere.
\end{proof}

\section{An additive to composition homomorphism}\label{secC5}
\setcounter{equation}{0}

In this section the following schema is used. The function $g(z) : \mathbb{C} \to \mathbb{C}$ is an entire function. The set $\mathcal{S}\subset\mathbb{C}$ is a domain. The functions $p(s):\mathcal{S}\to\mathbb{C}$ and $q(s):\mathcal{S}\to\mathbb{C}$ are holomorphic functions. The arc $\gamma:[a,b] \to \mathcal{S}$ is continuously differentiable.

The idea of this section is to derive a homomorphism from the integrand to the integral. That in, the following identity is derived,

\[
\int_\gamma \big{(}p(s) + q(s)\big{)}g(z)\,ds\bullet z = \int_\gamma p(s)g(z) ds\bullet \int_\gamma q(s)g(z)\,ds\bullet z\\
\]

This result will serve as the backbone for the rest of this paper. Everything that follows can be thought of as an attempt to generalize this.

The proof of this fact requires a trill of notation. The difficulty being, we must illustrate the notational trill well. As in the last section we can reduce this problem to a problem in integration, rather than a problem in contour integration. In this nature, there are two things which need to be shown. If,

\[
A = \int_a^b q(x)\,dx\\
\]

Then,

\[
\int_a^b q(x) g(z)\,dx\bullet z = \int_0^A g(z)\,du\bullet z\\
\]

Which follows by making the linear substitution $du = q(x)dx$. The second fact being that, if $d-c = d' - c'$, then,

\[
\int_c^d g(z)\,du\bullet z = \int_{c'}^{d'}g(z)\,du\bullet z\\
\]

This follows by looking at the partial compositions of this integral. Which is to say,

\[
\int_c^d g(z)\,du\bullet z = \lim_{\Delta u_j \to 0} \OmSum_{j=0}^{n-1} z + g(z)\Delta u_j\bullet z\\
\]

But $\Delta u_j = \Delta u_j'$ where $u_j'$ is a partition of $[c',d']$. This should be fairly intuitive. Therefore, if $B = \int_a^b p(x) + q(x)\,dx$ then $B-A = \int_a^b p(x)\,dx$, so,

\[
\int_a^b p(x) g(z)\,dx\bullet z = \int_A^B g(z)\,du\bullet z\\
\]

And of this, necessarily,

\[
\int_a^b p(x)g(z)\,dx\bullet \int_a^b q(x)g(z)\,dx\bullet z = \int_A^B g(z)\,du\bullet \int_0^A g(z)\,du\bullet z = \int_0^B g(z)\,du\bullet z\\
\]

However, by the above analysis,

\[
\int_0^B g(z)\,du\bullet z = \int_a^b \big{(}p(x) + q(x)\big{)}g(z)\,dx\bullet z\\
\]

With that, we've arrived at a homomorphism. Switching our discussion to contours is no different than the above, excepting we multiply by a factor of $\gamma'(x)$. This leads us to the next theorem detailing a homomorphic property of the compositional integral. It highlights the importance of separability in the compositional scenario. It works very similarly to the multiplicative property of the exponential, of which is a subset, when $g(z) = z$.

\begin{theorem}[Additive Homomorphism Theorem]\label{thmHOM}
Suppose $g(z) : \mathbb{C} \to \mathbb{C}$ is an entire function. Let $\mathcal{S}$ be a domain in $\mathbb{C}$. Suppose $p(s) : \mathcal{S} \to \mathbb{C}$ and $q(s) : \mathcal{S} \to \mathbb{C}$ are holomorphic functions. Let $\gamma : [ a,b] \to \mathcal{S}$ be a continuously differentiable arc. Then,

\[
\int_{\gamma} \big{(}p(s) + q(s)\big{)}g(z)\,ds\bullet z = \int_\gamma p(s)g(z)\,ds\bullet \int_\gamma q(s)g(z)\,ds\bullet z\\
\]
\end{theorem}

A vast majority of the rest of this paper will be focused on expanding this result. We want this result to hold not just for when $\phi(s,z) = p(s)g(z)$ is separable; but also in more extravagant situations. In order to do this though, we'll have to jump head first into a fair amount of technical analyses.

\section{Discussions of the residual term}\label{secC6}
\setcounter{equation}{0}

In this section we will have a brief discussion of what happens when the integrand of our contour integral has poles. In such a sense, we ask if there is something like a residue theorem. In order to do this, we will work on some simple cases; but the main integrand we are interested in has the form $\dfrac{\phi(s,z)}{s-\zeta}$. The author aims to carve out an approach to evaluating closed contours about this integrand. And in the process setup the development of convenient manners of evaluating these contour integrals.

Throughout this section the following schema is used. The set $\mathcal{S} \subset \mathbb{C}$ is a simply connected domain. The function $\phi(s,z) : \mathcal{S} \times \mathbb{C} \to \mathbb{C}$ is a holomorphic function. Additionally, $\gamma$ will always be used to denote a closed contour; and as per our custom it will be assumed to be continuously differentiable. Without confusion we will identify $\Delta \gamma_j$ with $\gamma'(x_j^*)\Delta x_j$. Both forms of the limit are equivalent due to the differential relationship our integral satisfies; and the condition that $\gamma$ is continuously differentiable.

We'll start with some basic linear cases first. From there, we'll attract ourselves to the favored example $\phi(s,z) = z^2$ and finish off with $\phi(s,z) = p(s)z^n$ for some holomorphic $p$. This section will take a bit longer than the last sections. We're going to focus on the lead which brings us to the next chapter. Of this, we'll go slow. Everything that's happened so far is to be a preamble to the thesis of this paper.\\

To begin, we take the exponential case $\phi(s,z) = z$ and let $\gamma$ be the unit circle about zero. The expression we are interested in is, for $|\zeta| < 1$,

\[
Y_\gamma(z) = \int_\gamma \frac{z}{s-\zeta} \,ds \bullet z
\]

To begin we analyse the partially nested compositions. Let $\{x_j\}_{j=0}^{n}$ be a partition of $[0,2\pi]$ in descending order and let $x_{j+1} \le x_j^* \le x_j$. The partial compositions take the form,

\begin{eqnarray*}
Y_n(z) &=& \OmSum_{j=0}^{n-1} z + z\frac{ie^{ix_j^*}}{e^{ix_j^*}-\zeta}\Delta x_j \bullet z\\
&=& \OmSum_{j=0}^{n-1} z(1 + \frac{ie^{ix_j^*}}{e^{ix_j^*}-\zeta}\Delta x_j) \bullet z\\
&=& z \cdot \prod_{j=0}^{n-1} (1 + \frac{ie^{ix_j^*}}{e^{ix_j^*}-\zeta}\Delta x_j)\\
\end{eqnarray*}

In this particular instance we can show the product on the right converges to $1$. Taking logarithms and using the asymptotic $\log(1+\Delta) \sim \Delta$ for small $\Delta \to 0$, we are given the equivalent form,

\[
\sum_{j=0}^{n-1}\frac{ie^{ix_j^*}}{e^{ix_j^*}-\zeta}\Delta x_j \to \int_\gamma \frac{ds}{s - \zeta} = 2\pi i
\]

Therefore, exponentiating this gives us a startling result,

\[
\int_\gamma \frac{z}{s-\zeta}\,ds \bullet z = z
\]

If we add a separable term to this function then our residual term is a little stranger. Let $\phi(s,z) = p(s) z$ for some holomorphic function $p$. Again, for simplicity, we will take our contour to be the unit disk. The expression we are interested in is,

\[
Y_\gamma(z) = \int_\gamma \frac{p(s)z}{s-\zeta}\,ds \bullet z
\]

Again, looking at the partial compositions, we get the expression,

\[
Y_n(z) = z \prod_{j=0}^{n-1} (1 + \frac{ip(e^{ix_j^*})e^{ix_j^*}}{e^{ix_j^*} - \zeta}\Delta x_j)
\]

In the limit, after making the same $\log(1 + \Delta) \sim \Delta$ asymptotic equivalence; this becomes the expression,

\[
Y_\gamma(z) = z e^{\displaystyle\int_\gamma \frac{p(s)}{s-\zeta}\,ds} = z e^{\displaystyle 2 \pi i p(\zeta)}
\]

We can generalize this result too, to the idea that,

\[
\int_\gamma \frac{p(s)z}{(s-\zeta)^{k+1}}\,ds \bullet z = z e^{\displaystyle 2 \pi i \frac{p^{(k)}(\zeta)}{k!}}
\]

This should hint to the idea that there is something like a residual function. In general, it will not always have such a nice form. In fact, we won't really get a single function. The above is an exceptionally nice form especially because; where $r$ is also a holomorphic function,

\[
\int_\gamma \frac{p(s)z}{(s-\zeta)^{k+1}}\,ds \bullet \int_\gamma \frac{r(s)z}{(s-\zeta)^{k+1}}\,ds \bullet z = \int_\gamma \frac{(p(s)+r(s))z}{(s-\zeta)^{k+1}}\,ds \bullet z
\]

This convenience arises when $\phi(s,z) = p(s)z$ is a linear transform in $z$; and in the next chapter arises more powerfully. This additive homomorphism has already made an appearance, and will make quite a few more. This motivates the idea of expanding the contour integral of $\dfrac{\phi(s,z)}{s-w}$  into a composition of functions $\dfrac{\phi(s,z)}{(s-\zeta)^{k+1}}$, and reducing it to an infinite composition of functions. A form of this relationship will become important in the next chapter.

The next step then is to look at the case where $\phi(s,z)$ is linear in $z$. By which we let $\phi(s,z) = p(s) z + q(s)$. This proves to be a difficult case because we no longer have a product identity as above. We do get an identity, but it looks a bit more cumbersome. The quantity we are interested in is,

\[
Y_\gamma(z) = \int_\gamma \frac{p(s)z + q(s)}{s-\zeta}\,ds \bullet z
\]

Then the partial compositions have the form,

\begin{eqnarray*}
Y_n(z) &=& z\prod_{j=0}^{n-1} (1 + \frac{ip(e^{ix_j^*})e^{ix_j^*}}{e^{ix_j^*} - \zeta}\Delta x_j)\\ &+&  i\sum_{j=0}^{n-1} \frac{q(e^{ix_j^*})e^{ix_j^*}}{e^{ix_k^*} - \zeta}\Delta x_j\prod_{k=0}^{j-1} (1 + \frac{ip(e^{ix_k^*})e^{ix_k^*}}{e^{ix_k^*} - \zeta}\Delta x_k)\\
\end{eqnarray*}

The first term is something we already know the value of, namely $e^{2\pi i p(\zeta)}$. The more difficult question is what the second term expands as. Using tools from the calculus of differential equations it is possible to derive a closed form expression, 

\[
Y_\gamma(z) = ze^{\displaystyle 2 \pi i p(\zeta)} + \int_\gamma \frac{q(w)}{w-\zeta}e^{\displaystyle \int_{\gamma_w} \frac{p(s)}{s-\zeta}\,ds}\,dw\\
\]

Here $\gamma_w$ is the arc $\gamma$ up to the point $w$. This expression serves to describe the status quo, and acts as another anomaly. This $Y_\gamma$ strictly depends on how we take the curve $\gamma$. No longer do we arrive at a nice clean and simple formula with no mention of $\gamma$--thanks to the exponent in the integral. The integral $Y_\gamma$ depends on $\gamma$; and not just the singularities within it.

What we aim to show in this section may not appear as though it's related to the above discussion. It may not be quite obvious where the argument comes from; or what the argument is for; but we want two things of which we can speak of about the residual function. These two things will form the basis of the next chapter. Our residual must be independent of $\gamma$; and in the best situation, should only depend on the residue and the function itself. 

This is directly in contradiction to what we just said, but it is possible through a more calculated approach, and quite a bit of legwork--algebraic legwork. It's difficult to express what we mean by this, so we'll try to be more illustrative.\\

To straighten things out; we'll start with the case $\phi(s,z) = z^2$. Of this nature, the integral we are interested in is,

\[
Y_\gamma(z) = \int_\gamma \frac{z^2}{s-\zeta}\,ds\bullet z\\
\]

We'll let $g( w,z)$ be the anti-derivative in question. Which, by careful analysis, can be reduced to a closed form expression. To that end, we can write,

\[
g( w,z) = \int_{\gamma_ w} \frac{z^2}{s-\zeta}\,ds\bullet z\\
\]

Where $\gamma_ w : [a,x] \to \mathcal{S}$ is the contour $\gamma$ up to the point $w$. This is simply a quick way of writing that $g(\gamma(a),z) = g(\alpha,z) = z$, and,

\[
\frac{d}{d  w} g( w,z) = \frac{g( w,z)^2}{ w - \zeta}
\]

However, there is a closed form expression for this thing. Namely,

\[
g( w,z) = \frac{1}{\displaystyle \frac{1}{z} + \log\dfrac{\alpha-\zeta}{ w - \zeta}}\\
\]

Therefore, in order to derive the value of $Y_\gamma$, we need only trace $g( w,z)$'s path as $w$ goes along $\gamma$ and back to the starting point $\alpha$. In doing so, we collect a remnant of $2\pi i$ but the $\log$'s cancel out. This is to say our desired integral equals,

\[
\int_\gamma \frac{z^2}{s-\zeta}\,ds\bullet z = \frac{1}{\displaystyle \frac{1}{z} - 2 \pi i}\\
\]

What is to be noticed from this; this integral is entirely independent of $\gamma$. This is because our integrand is separable. The separability allows us to apply The Additive Homomorphism Theorem \ref{thmHOM}. This is what it's good for. We're going to make a brief segue here, and introduce what will be the study of the next chapter. Supposing we have some holomorphic function $p(s)$, then,

\[
\int_\gamma \frac{p(s)z^2}{s-\zeta}\,ds\bullet z = \frac{1}{\displaystyle \frac{1}{z} - 2 \pi i p(\zeta)}\\
\]

By a similar analysis to the above. In that,

\[
\int_{\gamma_ w} \frac{p(s)z^2}{s-\zeta}\,ds\bullet z = \frac{1}{\displaystyle \frac{1}{z} - \int_{\gamma_ w}\frac{p(s)}{s-\zeta}\,ds}
\]

And once we complete the contour, $ w = \alpha$, the integral reduces to the above. And just as generally,

\[
\int_\gamma \frac{p(s)z^2}{(s-\zeta)^{k+1}}\,ds\bullet z = \frac{1}{\displaystyle \frac{1}{z} - \frac{2 \pi i}{k!}p^{(k)}(\zeta)}\\
\]

This is quite almost exactly what happens in the linear case. But, it's difficult to map out the phenomena, or why this is happening; or what is happening. But our contour integral makes absolutely no mention of $\gamma$ on the Right Hand Side. All that remains is some kind of residue. And just like in The Additive Homomorphism Theorem \ref{thmHOM}; if the reader care to notice; we get a composition identity. Suppose that $k,l \ge 0$, $p(s)$ and $q(s)$ are arbitrary holomorphic functions. Then,

\[
\int_\gamma \frac{p(s)z^2}{(s-\zeta)^{k+1}} + \frac{q(s)z^2}{(s-\zeta)^{l+1}}\,ds\bullet z =\int_\gamma \frac{p(s)z^2}{(s-\zeta)^{k+1}}\, ds \bullet  \int_\gamma \frac{q(s)z^2}{(s-\zeta)^{l+1}}\,ds\bullet z\\
\]

If we pose a more complex residue; a similar result arises. That is,
 
\[
\int_{\gamma} \frac{p(s)z^n}{(s-\zeta)^{k+1}}\,ds\bullet z = \frac{1}{\displaystyle\sqrt[n-1]{\frac{1}{z^{n-1}} +(1-n) \frac{2\pi i}{k!}p^{(k)}(\zeta)}}\\
\]

This can be arrived at from the equation,

\[
\frac{d}{d  w}g( w,z) = \frac{p( w)g( w)^n}{( w - \zeta)^{k+1}}\\
\]

We must be careful in how we take the root, but it doesn't pose much of a problem for local $z$. The thing to take away, is that this again is independent of the path $\gamma$. Now this won't always happen (these functions are separable so it does). But we would like for it to always happen.  So although there is certainly no additive homomorphism for arbitrary $\phi$ and arbitrary $\gamma$; when $\phi$ is arbitrary and $\gamma$ is closed we hope to reclaim something. To do so we'll have to broach deeper waters of where abstract algebra cools complex analysis.

The idea of this section is to setup the next chapter. And init, we want to setup the conditions in which,

\[
\int_\gamma \big{(}f(s) + g(s)\big{)}\phi(s,z)\, ds\bullet z = \int_\gamma f(s)\phi(s,z)\,ds\bullet\int_\gamma g(s)\phi(s,z)\,ds\bullet z\\
\]

In order to do this we need to show that these integrals are \textit{some kind} of independent of $\gamma$--if they aren't truly independent. That being the driving force behind the above identities. The rest of this section is devoted to proving that ${\displaystyle \int_\gamma \frac{\phi(s,z)}{s-\zeta}\,ds\bullet z}$ is independent of the path $\gamma$ \emph{upto conjugation}. This allows us to assign a value to this integral independent of the path once we take the quotient by an equivalence class. We will go about and call this quantity the residual class. The title of the next theorem is intended to highlight what the significance of the theorem is for.

\begin{theorem}[The Conjugate Principle of Residuals]\label{thmWllDef}
Let $\mathcal{S}\subset\mathbb{C}$ be a simply connected domain. Let $f(s,z) : \mathcal{S} \times \mathbb{C} \to \widehat{\mathbb{C}}$ be a meromorphic function with an isolated singularity $\zeta \in \mathcal{S}$. If $\gamma:[a,b] \to \mathcal{S}$ and $\varphi:[a,b] \to \mathcal{S}$ are two Jordan curves about the singularity $\zeta$ such that:

\begin{itemize}
	\item Both are oriented the same,\\
	\item The only poles enclosed within each is $\zeta$,\\
\end{itemize}

Then,

\[
\int_\gamma f(s,z)\,ds\bullet z \simeq \int_\varphi f(s,z)\,ds\bullet z\\
\]

Are conjugate similar by some integral $\int_\tau f(s,z) ds\bullet z$ for an arc $\tau: [a,b] \to \mathcal{S}$.
\end{theorem}

\begin{remark}
The author would like to put a `Bourbaki curvy road ahead' sign at this juncture. The following proof requires a fair amount of creativity to appreciate its significance. We've only chipped away at some of the intuition; as to why we might get something like this result. But we've been a tad opaque as to why it matters. The author figured diving right in may be more appropriate. The author thinks the curtness of this proof is appreciated.
\end{remark}

\begin{proof}
Starting with the contour integral of interest,

\[
Y_\gamma(z)= \int_\gamma f(s,z)\,ds\bullet z\\
\]

Here $\gamma:[a,b] \to \mathcal{S}$ and $\gamma(a) = \alpha = \gamma(b)$. By The Path Independence Theorem \ref{thmPTHIND}, the value ${\displaystyle g( w,z) = \int_{\alpha}^w f(s,z)\,ds\bullet z}$ depends only on the initial point $\alpha$ and $w$ (so long it doesn't pass around or through the singularity $\zeta$). So the final contour by consequence could only depend on $\alpha$. The object of this proof is to show it doesn't up to conjugation.

Take $\varphi: [a,b] \to \mathcal{S}$ such that $\varphi(a) = \beta = \varphi(b)$. Then, by The Path Independence Theorem \ref{thmPTHIND},

\[
Y_\varphi = \int_\varphi f(s,z)\,ds\bullet z\\
\]

Depends only on $\beta$. Call $\tau : [0,1] \to \mathcal{S}$ an arc such that $\tau(0) = \alpha$ and $\tau(1) = \beta$, such that $\tau$ doesn't encircle any poles or such. Then,

\[
\int_\tau \bullet \int_\gamma \bullet \int_{\tau^{-1}} \bullet \int_{\varphi^{-1}} \,ds\bullet z = z\\
\]

This follows from The Compositional Integral Theorem \ref{thmC3A}. To explicate we need to dig a tad deeper, consider the annulus $\mathcal{A}_\rho = \{s \in \mathcal{S}\,\mid\, \delta < |s-\zeta| < \delta + \rho\}$. Observe if we add a cut $\tau$ along the interior of $\mathcal{A}_\rho$; then $\mathcal{A}_\rho$ with this cut is simply connected. Then by The Compositional Integral Theorem \ref{thmC3A},

\begin{eqnarray*}
\Big{(}\int_{\partial \mathcal{A}_\rho} + \int_\tau - \int_\tau\Big{)}f(s,g(s,z))\,ds &=& 0\\
\Big{(}\int_{\partial \mathcal{A}_\rho} + \int_\tau - \int_\tau\Big{)}\frac{dg}{ds}\,ds &=& 0\\
\Big{(}\int_{\partial \mathcal{A}_\rho} + \int_\tau - \int_\tau\Big{)}dg &=& 0\\
\end{eqnarray*}

Since $\mathcal{A}_\rho$ with a cut along it is a simply connected domain. The result follows. By The Path Independence Theorem \ref{thmPTHIND} the arcs $\gamma$ and $\varphi$ are homotopically equivalent to arbitrary Jordan curves which touch their endpoints. Which concludes the proof.
\end{proof} 

Although we've only done this proof for $\gamma$ assumed to be of winding number $1$ (and we've implicitly assumed positive orientation), both of these conditions can be dropped throughout the proof. But in being illustrative we artfully avoided such cases. It is of mention though, if $\gamma_n = \gamma \bullet \gamma \bullet ... \bullet \gamma$ has winding number $n$ about $\zeta$, then,

\[
\int_{\gamma_n} f(s,z)\,ds\bullet z = \OmSum_{j=1}^n \int_{\gamma} f(s,z)\,ds\bullet z\\
\]

Which is to say, the more times we wind around the more we compose these objects with themselves. This works exactly how it works in Cauchy's analysis. Except in Cauchy's analysis we are given $z+2\pi i n \Res$. This is still composition, but the winding number reduces to a multiplicative factor, rather than an iterate.

In closing this chapter we use this theorem to create a definition. We will formally define what The Residual Class is. The next chapter will focus entirely on exploiting this definition and the stock-pile of knowledge at our disposal.

\begin{dfn}\label{defRSD}
The residual class $\Rsd(f,\zeta;z)$ of a meromorphic function $f: \mathcal{S} \times \mathbb{C} \to \widehat{\mathbb{C}}$ at a pole $\zeta$ is defined to be:

\[
\Rsd(f,\zeta;z) \ni \int_\gamma f(s,z)\,ds\bullet z\\
\]

Where:

\begin{itemize}
	\item $\gamma:[a,b] \to \mathcal{S}$ is a Jordan curve oriented positively.
	\item $\zeta$ is the only pole enclosed within $\gamma$.
\end{itemize}
\end{dfn}

And of which, our definition works to define a class of functions that are conjugate similar. Now, if we mod out by this equivalence relation $\simeq$, our resulting residual does not depend on $\gamma$. This effectively bypasses the problems we would have by our residuals depending on $\gamma$. This is a problem Cauchy never had to face; but lo, we face it now. But, how do we integrate in this strange space? How do we, as Cauchy so effortlessly did, say that ${\displaystyle \int_\gamma f = \sum \Res f }$ when our contours are only equivalent up to conjugation? How in heaven do we derive Cauchy's form of Taylor's Theorem? And where does Fourier play in all of this?

\chapter{Additive Properties Of Closed Contours}\label{chap3}

\section{Introduction}\label{secAT1}
\setcounter{equation}{0}

This chapter shall be split into two topics. The first topic has, as its cousin, the study of sums and series in Complex Analysis. Where, to express a holomorphic function $f(s): \mathcal{S} \to \mathbb{C}$ as a sequence of summands $f_n(s) : \mathcal{S} \to \mathbb{C}$ such that,

\[
f(s) = \sum_{n=0}^\infty f_n(s)\\
\]

Of these sums, the manner of convergence we shall follow was spearheaded by Reinhold Remmert in his two part text book series \emph{Theory of Complex Functions} \cite{Rem} and \emph{Classical Topics in Complex Function Theory} \cite{Rem2}. The notable property of how Remmert handles these sums being his consistent use of Normal Convergence. Therein, the sum above converges normally if,

\[
\sum_{n=0}^\infty ||f_{n}(s)||_{\mathcal{S}} = \sum_{n=0}^\infty \sup_{s \in \mathcal{S}}|f_{n}(s)| < \infty\\
\]

This confirms $f$'s holomorphy on $\mathcal{S}$. But additionally, confirms that all rearrangements of the sequence $f_n$ converge to $f$ as well. This is a nice convenience we'll use surreptitiously. One should think of this compactly though. We need not take the supremum norm on all of the domain $\mathcal{S}$, but solely its compact subsets. On seulement devoir que: for all compact sets $\mathcal{B} \subset \mathcal{S}$ the sum $\sum_{k=0}^\infty ||f_k(s)||_{s \in \mathcal{B}} < \infty$.

As said, the first topic of this paper is only related as a cousin to this subject. Of a confusing character, we will be interested in Infinite Compositions--not Infinite Sums. Infinite Compositions include Infinite Sums, and Infinite products. However, they also include strange beasts like the following,

\[
F(s,z) = \OmSum_{j=1}^\infty z + s^{2^j}z^2\bullet z\\
\] 

Which is holomorphic for $|s| < 1$ and $z \in \mathbb{C}$, and satisfies the petrifying functional equation,

\[
F(\sqrt{s},z) = F(s,z) + sF(s,z)^2\\
\]

Expressions like the above were well handled in the paper \cite{Nix2}. In this chapter we'll deal with much more complex beasts than that above. The idea of the first topic is to look at expressions of the form ${\displaystyle \int\sum_n f_n(s)\phi(s,z)\,ds\bullet z}$, and turn them into expressions of the from ${\displaystyle\OmSum_n\int f_n(s)\phi(s,z)\,ds\bullet z}$. At least, we want a zoology which looks something like this. In the separable case it looks exactly like this; so we are a bit hopeful. Strict care must be taken in order to do this. 

The first topic takes as its goal an analogue of Taylor Series. We shall extrapolate an involved proof which allows for a kind of Compositional Taylor Series. But, to get there, we'll provide a more general analysis of holomorphic functions as Infinite Compositions. And describe a strong meaning in the symbols ${\displaystyle \int \sum_n f_n(s)\phi(s,z)\,ds\bullet z = \OmSum_n \int f_n(s)\phi(s,z)\,ds\bullet z}$.

This will act as an extension to The Additive Homomorphism Theorem \ref{thmHOM}. It works in general instances and allows us to interchange infinite compositions through the integral, turning them into sums. 

The second topic is inter-weaved with the first topic. We'll need to add a good amount of algebraic analysis; and in doing so, stretch our theory to its limit. The idea of a Residual Class of a contour integral will be explored under the branch of algebraic analysis. We'll have to re-examine what the integral represents; and develop a modification to our integral type.

This chapter is intended to strengthen all the results we have just proven. In this nature, the proofs shall be rather short and sweet; thanks to the language we've developed; but will border on rather advanced concepts. This will only require abstracting the method of proof and washing it in generality to see its full effect. The culmination of this section is an attempt at generalizing The Additive Homomorphism Theorem \ref{thmHOM}; and exploiting it relentlessly.

\section{A brief expos\'{e} of Taylor Series}\label{secATP}

To start our work for this chapter, we're going to insinuate the main result of this chapter with some simple cases. From this we can better abstract the goal of this chapter. So for this instance, we are going to concern ourselves with separable functions. Through out this section we will write $\phi(s,z) = p(s) g(z)$. Here, $p(s):\mathcal{S} \to \mathbb{C}$ is a holomorphic function, and $g(z) : \mathbb{C} \to \mathbb{C}$ is an entire function.

We'll restrict ourselves to Jordan curves $\gamma : [a,b] \to \mathcal{S}$. As a motivating example we'll write the following. Let $| w-\zeta| < |\gamma(x)-\zeta| = |s-\zeta|$, and observe,

\[
\sum_{k=0}^\infty \frac{p(s)( w - \zeta)^k}{(s-\zeta)^{k+1}} = \frac{p(s)}{s- w}\\
\]

And by The Additive Homomorphism Theorem \ref{thmHOM}, we derive that,

\begin{eqnarray*}
\OmSum_{k=0}^\infty \int_\gamma \frac{p(s)g(z)( w-\zeta)^k}{(s-\zeta)^{k+1}}\,ds\bullet z &=& \int_\gamma \sum_{k=0}^\infty \frac{p(s)g(z)( w - \zeta)^k}{(s-\zeta)^{k+1}}\,ds\bullet z\\
&=& \int_\gamma \frac{p(s)g(z)}{s- w}\,ds\bullet z\\
\end{eqnarray*}

This needs to be done carefully; we should justify the interchange of the limit. But ignoring this caveat briefly, we can see how infinite compositions can be interchanged through the integral; turning into sums. To highlight what this thing is, it helps to provide context with examples.

Now in the separable case we need not require that $\gamma$ is a Jordan Curve, but for the nature of the general case: $\gamma$ needs to be closed, and for convenience a Jordan curve. And for illustration we continue to assume $\gamma$ is a Jordan curve. Consider the function $\phi(s,z) = p(s)$ when it is constant in $z$. There is something very familiar about this case, but we can write it strangely as,

\begin{eqnarray*}
\OmSum_{k=0}^{\infty} \int_\gamma \frac{p(s)( w - \zeta)^k}{(s-\zeta)^{k+1}}\,ds\bullet z &=& \OmSum_{k=0}^\infty \Big{(}z + \int_\gamma \frac{p(s)( w - \zeta)^k}{(s-\zeta)^{k+1}}\,ds\Big{)}\bullet z\\
&=& z + \sum_{k=0}^\infty \int_\gamma \frac{p(s)( w - \zeta)^k}{(s-\zeta)^{k+1}}\,ds\\
&=& z + 2 \pi i \sum_{k=0}^\infty \frac{p^{(k)}(\zeta)}{k!}( w - \zeta)^k\\
&=& z + 2 \pi i p( w)\\
&=& z + \int_\gamma \frac{p(s)}{s- w}\,ds\\
&=&\int_\gamma \frac{p(s)}{s- w}\,ds \bullet z\\
\end{eqnarray*}

So, for the constant case, our compositional identity will actually be the usual concept of a Taylor Series. We also have this result if $\phi(s,z) = zp(s)$; which follows from the formula for the residual. Which is to mean,

\begin{eqnarray*}
\OmSum_{k=0}^\infty \int_\gamma \frac{zp(s)( w - \zeta)^k}{(s-\zeta)^{k+1}}\,ds\bullet z &=& \OmSum_{k=0}^\infty ze^{\displaystyle\frac{2 \pi i}{k!}( w - \zeta)^kp^{(k)}(\zeta)}\bullet z\\
&=& z \prod_{k=0}^\infty e^{\displaystyle\frac{2 \pi i}{k!}( w - \zeta)^kp^{(k)}(\zeta)}\\
&=&z e^{\displaystyle\sum_{k=0}^\infty\frac{2 \pi i}{k!}( w - \zeta)^kp^{(k)}(\zeta)}\\
&=& z e^{\displaystyle2 \pi i p( w)}\\
&=& \int_\gamma \frac{zp(s)}{s- w}\,ds\bullet z\\
\end{eqnarray*}

And lastly this result also arises for the more anomalous case $\phi(s,z) = p(s) z^2$. In that we get,

\begin{eqnarray*}
\OmSum_{k=0}^\infty \int_\gamma \frac{z^2p(s)( w - \zeta)^k}{(s-\zeta)^{k+1}}\,ds\bullet z &=& \OmSum_{k=0}^\infty \frac{\bullet z}{\displaystyle \frac{1}{z} - \frac{2 \pi i}{k!}( w - \zeta)^kp^{(k)}(\zeta)}\\
&=& \frac{1}{\displaystyle \frac{1}{z} - \sum_{k=0}^\infty \frac{2 \pi i}{k!}( w - \zeta)^kp^{(k)}(\zeta)}\\
&=& \frac{1}{\displaystyle \frac{1}{z} - 2 \pi i p( w)}\\
&=& \int_\gamma \frac{z^2p(s)}{s- w}\,ds\bullet z\\
\end{eqnarray*}

The motivating argument we are trying to make in this chapter is that this result continues to hold for arbitrary $\phi$--if it is sufficiently modified. This is the essence of what we're trying to do. In doing this we'll have to be careful and develop a good amount of machinery. We have wet our feet with a tad bit of algebra. Henceforth we'll need to speak more so in these terms.

We will prove a more general construct; of which Taylor sums are a subset; but the exemplary goal is to prove this Taylor expansion for arbitrary $\phi(s,z)$, not just when it's separable. This will require a strict remediation of Cauchy's idea of a residue; and exhuming the structure of residuals.

\section{The Residual Class Theorem}\label{secAT2}
\setcounter{equation}{0}

In this section the following schema is used. The set $\mathcal{S}$ is a simply connected domain in $\mathbb{C}$. The function $f(s,z):\mathcal{S}\times \mathbb{C} \to \widehat{\mathbb{C}}$ is a meromorphic function. We will also assume the singularities of $f(s,z)$ are independent of $z$; that they are isolated in $\mathcal{S}$. The arc $\gamma:[a,b] \to \mathcal{S}$ is a closed contour which doesn't pass through a pole of $f$.

This section is intended to decompose the expression,

\[
\int_\gamma f(s,z)\,ds\bullet z\\
\]

Into its individual residual functions. To do that is a bit of a mind bender; but follows closely to Cauchy's Residue Theorem. In fact, this construction will act as a direct generalization of Cauchy's analysis. To get there we'll have to exploit the similar nature of residual classes.

The residual comes close to Cauchy's idea of a residue, and is intertwined with it. Insofar as, if $f(s,z) = f(s)$ is constant in $z$, then,

\[
\Rsd(f,\zeta;z) = z + 2\pi i \Res_{s=\zeta} f
\]

But it also amounts to complicated beasts like,

\begin{eqnarray*}
\Rsd(p(s)z, \zeta;z) &=& ze^{\displaystyle 2 \pi i \Res_{s=\zeta}p(s)}\\
\Rsd(p(s)z^2, \zeta;z) &=& \frac{1}{\displaystyle \frac{1}{z} - 2 \pi i \Res_{s=\zeta}p(s)}\\
\Rsd(p(s)z^n, \zeta;z)&=&  \frac{1}{\displaystyle \sqrt[n-1]{\frac{1}{z^{n-1}} +(1-n) 2 \pi i \Res_{s=\zeta}p(s)}}\\
\end{eqnarray*}

But in general circumstances, no single clean formula can be derived like the above; and in most cases these things are not single elements. The class is a class of functions, and not a single function; it's just in the separable case they're a single function. It is a bit more of a mysterious thing than the residue. The residual class is actually a property of $f$, and bears little relation to how the integral is taken. 

But additionally, if there are multiple poles of $f$ within $\gamma$, then the integral doesn't depend on the path $\gamma$. It only depends on the residuals within the contour. By which it is meant, the following theorem,

\begin{theorem}\label{thmPTHINDP}
Suppose $\mathcal{S} \subset \mathbb{C}$ is a simply connected domain. Suppose $f(s,z) : \mathcal{S} \times \mathbb{C} \to \widehat{\mathbb{C}}$ is a meromorphic function. If $\gamma :[a,b] \to \mathcal{S}$ and $\varphi:[a,b]\to\mathcal{S}$ are Jordan oriented the same, contain the same singularities of $f$ in their interior then,

\[
\int_\gamma f(s,z)\,ds\bullet z \simeq \int_{\varphi}f(s,z)\,ds\bullet z\\
\]

Which is to say there is an arc $\tau \subset \mathcal{S}$ such that,

\[
\int_\tau f(s,z) \,ds \bullet \int_\gamma f(s,z)\,ds\bullet z = \int_{\varphi}f(s,z)\,ds \bullet \int_\tau f(s,z) \,ds\bullet z\\
\]
\end{theorem}

\begin{proof}
See the proof of The Well-Defined Principle of Residuals Theorem \ref{thmWllDef}. The proof is exactly the same.
\end{proof}

To begin the foray, we'll take two singularities of $f$, call them $\zeta_1, \zeta_2 \in \mathcal{S}$. Let $\tau_1$ be a circle about $\zeta_1$ oriented positively, with a hole in it of width $\delta$. Let $\tau_2$ be a circle about $\zeta_2$ oriented positively, with a hole in it of width $\delta$. Let $\rho^-$ be a line connecting the bottom part of each hole, oriented positively. Let $\rho^+$ be a line connecting the top part of each hole, oriented negatively. Call $\gamma = \tau_1 \bullet \rho^+ \bullet \tau_2 \bullet \rho^{-}$; and for the sake of the argument assume the only singularities within $\gamma$ are $\zeta_1$ and $\zeta_2$. Additionally call $h_1$ the arc covering the hole of $\tau_1$.

By careful analysis,

\[
\int_\gamma f(s,z)\,ds\bullet z = \int_{\tau_1} f(s,z)\,ds\bullet \int_{\rho^+} f(s,z)\,ds\bullet\int_{\tau_2} f(s,z)\,ds\bullet\int_{\rho^{-}} f(s,z)\,ds\bullet z\\
\]

And,

\[
\int_{\tau_1} f(s,z)\,ds\bullet \int_{h_1}f(s,z)ds\bullet z \simeq \Rsd(f,\zeta_1;z)
\]

And therefore,

\[
\int_\gamma f(s,z) \,ds\bullet z \simeq \Rsd(f,\zeta_1;z) \bullet \int_{h_1^{-1}}f(s,z)ds\bullet \int_{\rho^+} f(s,z)\,ds\bullet\int_{\tau_2} f(s,z)\,ds\bullet\int_{\rho^{-}} f(s,z)\,ds\bullet z\\
\]

However $h_1^{-1} \bullet \rho^+ \bullet \tau_2 \bullet \rho^{-}$ is a closed contour about the singularity $\zeta_2$. And therefore,

\[
\int_\gamma f(s,z) \,ds\bullet z \simeq \Rsd(f,\zeta_1;z) \bullet \Rsd(f,\zeta_2;z) \bullet z\\
\]

But, since the contour is independent of how we choose $\gamma$ (up-to conjugation), we can put the contour about $\zeta_2$ first. So the ordering of $\zeta_1$ and $\zeta_2$ doesn't matter, we get that,

\[
\int_\gamma f(s,z) \,ds\bullet z \simeq \Rsd(f,\zeta_2;z) \bullet \Rsd(f,\zeta_1;z) \bullet z\\
\]

We should briefly note that when we write $\Rsd(f,\zeta;z)$ we mean some representative of this set. It is not that for $g \in \Rsd(f,\zeta_1;z)$ and $h \in \Rsd(f,\zeta_2;z)$ that $ g\bullet h = h \bullet g$; but rather there is some $h' \in \Rsd(f,\zeta_2;z)$ such that $h' \bullet g = g \bullet h$. In group theoretic terms, we can think of this as $\Rsd$ being a normal set.

Now this theorem applies for as many singularities as are contained within the contour. And therefore, of this nature we get a residue theorem. This theorem is of the residual class though, so we'll label it The Residual Class Theorem.

\begin{theorem}[The Residual Class Theorem]\label{thmRES}
Suppose $\mathcal{S} \subset \mathbb{C}$ is a simply connected domain. Suppose $f(s,z) : \mathcal{S} \times \mathbb{C} \to \widehat{\mathbb{C}}$ is a meromorphic function. Let $\gamma : [a,b] \to \mathcal{S}$ be a Jordan curve oriented positively. Let $\{\zeta_j\}_{j=1}^n$ be a list of the poles of $f(s,z)$ contained in $\gamma$--done so in no particular order. Then,

\[
\int_\gamma f(s,z)\,ds\bullet z \simeq \OmSum_{j=1}^n \Rsd(f,\zeta_j;z)\bullet z\\
\]
\end{theorem}

We'll add a little segue here. It should be noted that this result holds when ${\displaystyle \frac{\partial \zeta_j}{\partial z} \neq 0}$, so long as there is no singularity along the path of $\gamma$ as $z$ varies. This is to say, along non-isolated singularities this result is still possible.

But this should be done very carefully; and in such instances $\Rsd(f,\zeta_j;z)$ is not necessarily a meromorphic function on all of $\mathcal{S} \times \mathbb{C}$--except perhaps if $\mathcal{S} = \mathbb{C}$. The domain in $z$ must be analyzed carefully when ${\displaystyle\frac{\partial \zeta_j}{\partial z} \neq 0}$. We will avoid such cases. But, through a more mediated analysis, the following proofs are adaptable to these situations. To keep this paper clear, the proofs that follow are intended for the cases where ${\displaystyle\frac{\partial \zeta_j}{\partial z}=0}$. Or that, our singularities are independent of $z$. Isolated singularities are our main currency.

This small little theorem will make up the basis of much of this chapter. A lot of the work we shall do will focus exclusively on consequences of this result. This nifty little tool helps form comprehension of something that isn't very easy to comprehend.

\section{The Congruent Integral}\label{secATE}
\setcounter{equation}{0}

One of the most powerful theorems in Cauchy's analysis is the path independence of integrals about singularities. If we take a Jordan curbe about a bunch of singularities, the only thing that matters are all the residues of the singularity contained within $\gamma$. As we've seen so far, we do not have this result--excepting in the separable case. However, we do get something close. 

We've defined the relation;

\[
\int_{\gamma} f(s,z)\,ds\bullet z \simeq \int_{\varphi} f(s,z)\,ds\bullet z\\
\]

If there is some arc $\tau:[0,1]\to\mathcal{S}$ such that,

\[
\int_\tau f(s,z) \,ds\bullet \int_\gamma f(s,z)\,ds \bullet z = \int_\varphi f(s,z)\,ds\bullet \int_{\tau}f(s,z)\,ds\bullet z\\
\]

If we call $\mathcal{P}$ the group of compositional integrals with integrand $f$, then we aim to consider now the group,

\[
\mathcal{A} = \mathcal{P} / \simeq\\
\]

Which is to define a topological group. In this space the class of functions defined by $\Rsd(f,\zeta;z)$ is no longer a class of functions, but a singular element. For convenience integrals done in this space shall be written,

\[
\oint_\gamma f(s,z)\,ds\bullet z\\
\]

By which the space $\mathcal{P}/\simeq$ is somewhat equivalent to the normal Cauchy analysis. Closed contours are equivalent depending on the singularities within (and their winding numbers). In this space $\oint$ behaves a lot like $\oint$ in its typical interpretation as a contour integral in Cauchy's analysis (the additive case). A je ne sais quoi feature of this contour integral is that, about the congruent integral, is that if $\sigma:[a,b] \to \mathcal{S}$ is a non-closed arc; and doesn't encircle any poles, then,

\[
\oint_\sigma f(s,z) \,ds\bullet z = \int_\sigma f(s,z) \,ds\bullet z\\
\]

The only time $\oint$ differs from $\int$, is when performed over a closed contour; or when wrapping around a pole. On top of this,

\[
\int_\sigma \bullet \oint = \oint \bullet \int_\sigma\\
\]

In which case, $\oint \,ds\bullet z$ sort of behaves like the usual integral $\int \,ds$ over a closed contour. Where now the fundamental statement of the congruent integral is that,

\begin{theorem}[The Fundamental Theorem Of The Congruent Integral]\label{thmFCINT}
Let $\mathcal{S}$ be a simply connected domain. Assume $f(s,z) : \mathcal{S} \times \mathbb{C} \to \widehat{\mathbb{C}}$ is a meromorphic function with isolated singularities. Let $\gamma, \varphi:[a,b] \to \mathcal{S}$ be Jordan curves oriented the same, such that both contain the same singularities of $f$; then,

\[
\oint_\gamma f(s,z) \,ds\bullet z = \oint_\varphi f(s,z)\,ds\bullet z\\
\]
\end{theorem}

We call this the fundamental theorem because one can do everything one does in Cauchy's analysis if one has this. Insofar, as one can force Taylor's theorem; and the interchange of sums with integrals. The only necessity being, when we make the pull-back to the compositional integral, that we add a layer of conjugations to each element. Much of the operations may appear gauche, but of course; the oddness of the object requires the oddness of the approach.

\section{Additivity Theorem of Contour Integration}\label{secAT3}
\setcounter{equation}{0}

In this section the following schema is used. The set $\mathcal{S}$ is a simply connected domain in $\mathbb{C}$. The function $\phi(s,z) : \mathcal{S} \times \mathbb{C} \to \mathbb{C}$ is holomorphic. The functions $f(s),g(s):\mathcal{S} \to \widehat{\mathbb{C}}$ are meromorphic functions. The curve $\gamma:[a,b] \to \mathcal{S}$ is a Jordan curve oriented positively.

The goal of this section is to turn sums into compositions. No less, we aim to take them back and forth isomorphically.  Compositions become additions; this is quite the anomaly. The goal is to extend The Additive Homomorphism Theorem \ref{thmHOM} into a more general result.

Therein, we aim to show,

\[
\oint_\gamma \big{(}f(s) + g(s)\big{)}\phi(s,z)\,ds\bullet z = \oint_\gamma f(s)\phi(s,z)\,ds\bullet \oint_\gamma g(s)\phi(s,z)\,ds\bullet z\\
\]

Now this result follows from The Residual Class Theorem \ref{thmRES} and a separate lemma. The description of this lemma is fairly straight forward. It has its place in Cauchy's analysis, but it's a bit of a triviality in that scenario. In our present situation we have to talk more strictly about what the congruent integral means in these situations.

Introduce the term $h(s) : \mathcal{S} \to \mathbb{C}$; which we assume to be holomorphic. Then the idea is,

\[
\int_\tau \bullet \int_\gamma \bullet  \int_{\tau^{-1}}\big{(}f(s) + h(s)\big{)}\phi(s,z)\,ds\bullet z = \int_\tau \bullet \int_\gamma \bullet \int_{\tau^{-1}}f(s)\phi(s,z)\,ds\bullet z\\
\]

Where each of these conjugations are done across different functions. So we can think of this being as an equality between classes, and not within the same class. Unless we were to rewrite a bit more comprehensively what the congruent integral means.

If we say that,

\[
A \simeq B\\
\]

If there exists some sequence of meromorphic function $u_j(s) : \mathcal{S} \to \widehat{\mathbb{C}}$, and a sequence of arcs $\tau_j:[0,1] \to \mathcal{S}$ such that,

\[
\OmSum_j \int_{\tau_j} u_j(s)\phi(s,z)\,ds\bullet A \bullet \MhSum_j \int_{\tau_j^{-1}} u_j(s)\phi(s,z)\,ds\bullet z = B\\
\]

Then, in this language, and once modding out by this equivalence relation, and defining a congruent integral in this manner, then:

\[
\oint_\gamma \big{(}f(s) + h(s)\big{)}\phi(s,z)\,ds\bullet z = \oint_\gamma f(s)\phi(s,z)\,ds\bullet z\\
\]

Which is really just a fancy way of keeping track of the conjugations. But without this fancy way, this paper would contain too many equations going off the page.

To understand the significance of this result will culminate in the desired result. But foregoing the significance at the moment; we'll prove it. For the sake of comprehension; we'll prove it with respect to the residual. This affords us a more atomic understanding.

\begin{lemma}\label{lmaAT4A}
Suppose $\mathcal{S} \subset \mathbb{C}$ is a simply connected domain. Let $\phi(s,z): \mathcal{S} \times \mathbb{C} \to \mathbb{C}$ be a holomorphic function. Let $f(s) : \mathcal{S} \to \mathbb{C}$ be a meromorphic function with a pole at $\zeta \in \mathcal{S}$. Let $h(s):\mathcal{S} \to \mathbb{C}$ be a holomorphic function in a neighborhood of $\zeta$.

\[
\Rsd\Big{(}\big{(}f(s) + h(s)\big{)}\phi(s,z), \zeta;z\Big{)} \simeq \Rsd \big{(}f(s)\phi(s,z), \zeta;z\big{)}\\
\]

Which is to say, there are two arcs $\tau_0,\tau_1 \subset\mathcal{S}$ and two functions $u_0, u_1 : \mathcal{S} \to \widehat{\mathbb{C}}$, such that if $A \in \Rsd (f+h)$ and $B \in \Rsd f$,

\[
\int_{\tau_0} \bullet \int_{\tau_1} \bullet A \bullet \int_{\tau_1^{-1}} \bullet \int_{\tau_0^{-1}} = B\\
\]

Where,

\[
\int_{\tau_i} = \int_{\tau_i}u_i(s)\phi(s,z)\,ds\bullet z\\
\]
\end{lemma}

\begin{proof}
Recalling that this residual does not depend on our closed contour $\gamma$ (upto conjugation), we can choose $\gamma(x) = \zeta + \delta e^{ix} : [0,2\pi] \to \mathcal{S}$; here $\delta > 0$ and is arbitrarily small. Then,

\[
\int_\gamma \big{(}f(s) + h(s)\big{)}\phi(s,z)\,ds\bullet z = \int_0^{2\pi} i\delta\big{(}f(\zeta + \delta e^{ix}) + h(\zeta + \delta e^{ix})\big{)} \phi(\zeta + \delta e^{ix},z)e^{ix}\,dx\bullet z\\
\]

Now,

\[
\delta h(\zeta + \delta e^{ix})\phi(\zeta + \delta e^{ix},z) e^{ix} \to 0 \,\,\text{as}\,\,\delta \to 0\\
\]

Which is an easy consequence of $h(s)\phi(s,z)$ being holomorphic in a neighborhood of $\zeta$. So in letting $\delta \to 0$, and interchanging the limit through the integral, we're all done. Interchanging the limit poses no problem; if we know the intent of how we want to interchange the limit. Take its partial compositions, and observe,

\begin{eqnarray*}
\OmSum_j z + i\delta e^{ix_j}\big{(}f(\zeta + \delta e^{ix_j}) + h(\zeta + \delta e^{ix_j})\big{)}\phi(\zeta + \delta e^{ix_j},z)\Delta x_j \bullet z -\\
- \OmSum_j z + i\delta e^{ix_j}f(\zeta + \delta e^{ix_j})\phi(\zeta + \delta e^{ix_j},z)\Delta x_j \bullet z = \mathcal{O}(\delta)\\
\end{eqnarray*}

This follows from basic properties of composition. Which is to say if $q_j(\delta,z) = (1 + \mathcal{O}(\delta)) p_j(z)$, then $\OmSum_j q_j(\delta,z)\bullet z = \OmSum_j p_j(z)\bullet z + \mathcal{O}(\delta)$. Remembering this is compactly normally convergent, so the $\mathcal{O}$ term is very nice. As $\Delta x_j \to 0$, and the fact the limits are independent of $\delta$; we must have,

\begin{eqnarray*}
\lim_{\Delta x_j \to 0} \OmSum_j z  + i\delta e^{ix_j}\big{(}f(\zeta + \delta e^{ix_j}) + h(\zeta + \delta e^{ix_j})\big{)}\phi(\zeta + \delta e^{ix_j},z)\Delta x_j \bullet z \simeq\\
\simeq \lim_{\Delta x_j \to 0}\OmSum_j z + i\delta e^{ix_j}f(\zeta + \delta e^{ix_j})\phi(\zeta + \delta e^{ix_j},z)\Delta x_j \bullet z\\
\end{eqnarray*}

Which concludes the proof.
\end{proof}

With this lemma, we're essentially at the desired result of this section. If $f(s)$ and $g(s)$ share no common poles then,

\[
\oint_\gamma \big{(}f(s) + g(s)\big{)}\phi(s,z)\,ds\bullet z = \oint_\gamma f(s)\phi(s,z)\,ds\bullet \oint_\gamma g(s)\phi(s,z)\,ds\bullet z\\
\]

As the residues of the combined function $f+g$ depend on $f$ and $g$ independently; and by Lemma \ref{lmaAT4A} $\Rsd((f+g)\phi,\zeta;z) \simeq \Rsd(f\phi, \zeta;z)$ if $\zeta$ is not a pole of $g$; and vice versa. The contour can be organized as we wish, which allows us to put the singularities of $f$ first and the singularities of $g$ second. So, the only issue we may have with this result is if $f$ and $g$ share poles. This can be better explained one step at a time. Assume $g(\zeta,z) = f(\zeta,z) = \infty$. The aim is to show that,

\[
\Rsd((f+g)\phi,\zeta;z) \simeq \Rsd(f\phi,\zeta;z)\bullet \Rsd(g\phi,\zeta;z)\bullet z\\
\]

Now this result is certainly true if we perturb $g$ to $\tilde{g}$ such that $\tilde{g}(\tilde{\zeta}) = \infty$ and $\tilde{g} \to g$ and $\tilde{\zeta} \to \zeta$. By continuity the problem is a write-off. Therefore we arrive at the next theorem.

\begin{theorem}[The Additivity Theorem]\label{thmADD}
Suppose $\mathcal{S} \subset \mathbb{C}$ is a simply connected domain. Suppose $f(s)$ and $g(s)$ are meromorphic functions which take $\mathcal{S} \to \widehat{\mathbb{C}}$. Let $\phi(s,z) : \mathcal{S} \times \mathbb{C} \to \mathbb{C}$ be holomorphic. Let $\gamma : [a,b] \to \mathcal{S}$ be a Jordan curve. Then,

\[
\oint_\gamma \big{(}f(s) + g(s)\big{)}\phi(s,z) \,ds\bullet z = \oint_\gamma f(s)\phi(s,z)\,ds\bullet \oint_\gamma g(s)\phi(s,z)\,ds\bullet z\\
\]
\end{theorem}

\begin{proof}
Let $\{\zeta_j\}_{j=1}^n$ be a list of the poles of $(f+g)\phi$ contained within $\gamma$. By The Residual Class Theorem \ref{thmRES} we know that,

\[
\int_\gamma \big{(}f(s) + g(s)\big{)}\phi(s,z) \,ds\bullet z \simeq \OmSum_{j=1}^n \Rsd\big{(}(f+g)\phi, \zeta_j;z\big{)}\bullet z
\]

By Lemma \ref{lmaAT4A}, if $\zeta_j$ is a pole solely of $f$ then $\Rsd((f+g)\phi, \zeta_j;z) \simeq \Rsd(f\phi,\zeta_j;z)$; similarly if $\zeta_j$ is a pole solely of $g$. If $\zeta_j$ is a pole of both then $\Rsd((f+g)\phi, \zeta_j;z) \simeq \Rsd(f\phi,\zeta_j;z) \bullet \Rsd(g\phi,\zeta_j;z) \bullet z$. All of these residuals commute (in the quotient) by organizing the contour any manner we see fit; and therefore the result follows.
\end{proof}

\section{Infinite Compositions of Contour Integrals}\label{secAT4}
\setcounter{equation}{0}

In this section the following schema is used. As before, the set $\mathcal{S}$ is a simply connected domain in $\mathbb{C}$. The sequence of functions $f_n(s):\mathcal{S} \to \widehat{\mathbb{C}}$ are meromorphic. Further, the functions $f_n(s)$ are normally summable on their domain of holomorphy (which is to say on its compact subsets). The function $f:\mathcal{S}  \to \widehat{\mathbb{C}}$ is used to designate the function,

\[
f(s) = \sum_{n=0}^\infty f_n(s)
\]

As in the last section $\gamma$ is a Jordan curve, $\gamma:[a,b] \to \mathcal{S}$. Additionally, as in all of the above; $\phi(s,z):\mathcal{S} \times \mathbb{C} \to \mathbb{C}$ is a holomorphic function. Accustoming ourselves with this we write what we aim to show in this section. It is little different than what we just wrote.

\[
\oint_\gamma f(s)\phi(s,z)\,ds\bullet z = \OmSum_{n=0}^\infty \oint_\gamma f_n(s)\phi(s,z)\,ds\bullet z\\
\]

And,

\[
\oint_\gamma f(s)\phi(s,z)\,ds\bullet z = \MhSum_{n=0}^\infty \oint_\gamma f_n(s)\phi(s,z)\,ds\bullet z\\
\]

Just as well the ordering of these infinite compositions is entirely arbitrary. Now, the work to do is quite a handful. To begin, it's required we show convergence of these expressions. In order to do this we'll play a little fast and loose, but with the exception of a detailed proof that if ${\displaystyle \sum_j ||\phi_j - z|| <\infty}$ is summable, then so is ${\displaystyle \sum_j ||h\bullet \phi_j \bullet h^{-1} - z|| <\infty}$--everything is sound. The author thinks this apparently true; it is of the reader to prove it if they don't take my word.

\begin{theorem}\label{thmCVG}
Suppose $\mathcal{S}$ is a simply connected domain in $\mathbb{C}$. Suppose $f_n(s) : \mathcal{S} \to \widehat{\mathbb{C}}$ is a sequence of meromorphic functions. Suppose these functions are normally summable on their domain of holomorphy. 

Let $\phi(s,z) : \mathcal{S} \times \mathbb{C} \to \mathbb{C}$ be a holomorphic function. Suppose $\gamma: [a,b] \to \mathcal{S}$ is a Jordan curve. Then the following infinite compositions converge uniformly on their domain of holomorphy,

\[
\OmSum_{n=0}^\infty \oint_\gamma f_n(s)\phi(s,z)\,ds\bullet z
\]

And,

\[
\MhSum_{n=0}^\infty \oint_\gamma f_n(s)\phi(s,z)\,ds\bullet z\\
\]
\end{theorem}

\begin{proof}
By The Compactly Normal Convergence Theorem \ref{ThmNormCvg}, the first thing we need is that,

\[
\sum_{n=0}^\infty \Big{|}\Big{|}\oint_\gamma f_n(s)\phi(s,z)\,ds\bullet z - z\Big{|}\Big{|}_{\mathcal{K}} < \infty
\]

But by the conjugation, this is equivalent to,

\[
\sum_{n=0}^\infty \Big{|}\Big{|}\int_\gamma f_n(s)\phi(s,z)\,ds\bullet z - z\Big{|}\Big{|}_{\mathcal{K}} < \infty
\]

Because if ${\displaystyle \sum_j ||\phi_j - z||_{\mathcal{K}} <\infty}$ for all $\mathcal{K}$, then ${\displaystyle \sum_j ||h \bullet \phi_j \bullet h^{-1} - z||_{\mathcal{K}} <\infty}$ for all $\mathcal{K}$. So long as $h$ is holomorphic and non-constant (in this case a contour integral).

For all compact disks $\mathcal{K}\subset \mathcal{H}$ where $\mathcal{H}$ is a domain of holomorphy for all $\int_\gamma f_n \phi$. There exists a larger compact set $\mathcal{K} \subset \mathcal{L}$ such that, by The Triangle Inequality Theorem \ref{thmTriInq},

\[
\Big{|}\Big{|}\int_\gamma f_n(s)\phi(s,z)\,ds\bullet z - z\Big{|}\Big{|}_{\mathcal{K}} \le \int_\gamma |f_n(s)|\big{|}\big{|}\phi(s,z)\big{|}\big{|}_{\mathcal{L}}\,|ds|\\
\]

This is certainly a summable series because $f_n$ is normally summable.
\end{proof}

The author has glossed over a few nefarious details in the above proof. But if one pays close attention to the language everything is right and in order. The domain of holomorphy in $z$ was treated fairly non-chalantly. But since we have isolated singularities for each $f_n$, the final result only has isolated singularities in $z$; and diverges on a set of measure zero.

Now from this convergence theorem, it really isn't too hard to see our main result. By The Additivity Theorem \ref{thmADD}, it's turned into a pulling the limit through the integral kind of problem.

\begin{theorem}[The Summation Theorem]\label{thmSUM}
Suppose $\mathcal{S}$ is a simply connected domain in $\mathbb{C}$. Suppose $f_n(s) : \mathcal{S} \to \widehat{\mathbb{C}}$ is a sequence of meromorphic functions. Suppose these functions are normally summable on their domain of holomorphy. Suppose $\gamma: [a,b] \to \mathcal{S}$ is a Jordan curve. Additionally, $\phi(s,z):\mathcal{S} \times \mathbb{C} \to \mathbb{C}$ is holomorphic.

Then the following infinite compositions converge uniformly on their domain of holomorphy to,

\[
\oint_\gamma \sum_{n=0}^\infty f_n(s) \phi(s,z)\,ds\bullet z = \OmSum_{n=0}^\infty \oint_\gamma f_n(s)\phi(s,z)\,ds\bullet z\\
\]

And,

\[
\oint_\gamma \sum_{n=0}^\infty f_n(s) \phi(s,z)\,ds\bullet z = \MhSum_{n=0}^\infty \oint_\gamma f_n(s)\phi(s,z)\,ds\bullet z\\
\]
\end{theorem}

\begin{proof}
Take the sequence of sums $F_N(s,z) = \sum_{n=0}^N f_n(s)$ then by The Additivity Theorem \ref{thmADD} we have that,

\[
\oint_{\gamma}F_N(s)\phi(s,z)\,ds\bullet z = \OmSum_{n=0}^N \oint_\gamma f_n(s)\phi(s,z)\,ds\bullet z\\
\]

The limit as $N\to\infty$ converges. What's needed is that,

\[
\lim_{N\to\infty}\oint_{\gamma}F_N(s)\phi(s,z)\,ds\bullet z = \oint_\gamma \lim_{N\to\infty}F_N(s)\phi(s,z)\,ds\bullet z\\
\]

To get it we need only follow a rather plain procedure. If $\mathcal{K} \subset \mathcal{H}$ where $\mathcal{H}$ is a domain of holomorphy for $\int_\gamma f_n(s)\phi(s,z)\,ds\bullet z$ for all $n>L$. For all $\epsilon>0$ there exists some $L$ such when $M > N > L$,

\[
\Big{|}\Big{|}\oint_\gamma \sum_{n=N+1}^M f_n(s)\phi(s,z)\,ds\bullet z - z\Big{|}\Big{|}_{\mathcal{K}} < \epsilon\\
\]

By the comparison, for some larger compact set $\mathcal{K}\subset\mathcal{L}$,

\[
\Big{|}\Big{|}\oint_\gamma \sum_{n=N+1}^M f_n(s)\phi(s,z)\,ds\bullet z - z\Big{|}\Big{|}_{\mathcal{K}} \le \sum_{n=N+1}^M\int_\gamma |f_n(s)|\big{|}\big{|}\phi(s,z)\big{|}\big{|}_{\mathcal{L}}\,|ds| < \epsilon\\
\]

Because,

\begin{eqnarray*}
\Big{|}\Big{|}\oint_\gamma F_M\phi - \oint_\gamma F_N\phi\Big{|}\Big{|}_{\mathcal{K}} &=& \Big{|}\Big{|}\oint_\gamma F_N\phi \bullet \oint_\gamma \sum_{n=N+1}^M f_n(s)\phi(s,z)  - \oint_\gamma F_N\phi\Big{|}\Big{|}_{\mathcal{K}}\\
&\le& \Big{|}\Big{|}\oint_\gamma F_N\phi \bullet (z+\epsilon)\bullet z  - \oint_\gamma F_N\phi\Big{|}\Big{|}_{\mathcal{K}}\\
&<& \epsilon'\\
\end{eqnarray*}

Recall that $\oint_\gamma F_N$ is a normal family; therefore shifting $z$ to $z+\epsilon$ produces an error of $\epsilon'$. By normal convergence everything works out. A similar procedure works just as well for the inverted orientation $\MhSum$; or by applying functional inversion. C'est,

\[
\Big{(}\OmSum_n \oint_\gamma h_n\Big{)}^{-1} = \MhSum_n \oint_{\gamma^{-1}} h_n\\
\]

Which is what we mean by: functional inversion derives the result for the inverse case. The details are left to the reader as an exercise.
\end{proof}

The reader should note that this provides a proof for The Compositional Taylor Theorem; but just that it's written a bit more abstractly. This was the direction of this theorem. We can write it rigorously now.

\begin{theorem}[The Compositional Taylor's Theorem]\label{thmC5A}
Let $\mathcal{S}$ be a simply connected domain in $\mathbb{C}$. Suppose $\phi(s,z) : \mathcal{S} \times \mathbb{C} \to \mathbb{C}$ is a holomorphic function. Let $\mathcal{U} \subset \mathcal{S}$ be an open disk about a point $\zeta \in \mathcal{S}$. Let $\gamma$ be a Jordan curve in $\mathcal{S}$ containing $\mathcal{U}$. For all $ w \in \mathcal{U}$ and for all $z \in \mathbb{C}$,

\[
\OmSum_{k=0}^\infty \oint_\gamma \frac{\phi(s,z)( w - \zeta)^k}{(s-\zeta)^{k+1}}\,ds\bullet z = \oint_\gamma \frac{\phi(s,z)}{s- w}\,ds\bullet z 
\]

And,

\[
\MhSum_{k=0}^\infty \oint_\gamma \frac{\phi(s,z)( w - \zeta)^k}{(s-\zeta)^{k+1}}\,ds\bullet z = \oint_\gamma \frac{\phi(s,z)}{s- w}\,ds\bullet z 
\]
\end{theorem}

If the reader wishes to see this in more concrete form; it's a tad difficult to visualize. But one can think of $\gamma_\delta, \varphi_\delta:[0,2\pi] \to \mathcal{S}$ such that $\gamma_\delta : x \mapsto \zeta + \delta e^{ix}$ and $\varphi_\delta : x \mapsto w + \delta e^{ix}$. Then what we are effectively saying is that,

\[
\lim_{\delta \to 0} \Big{(}\OmSum_{k=0}^\infty \int_{\gamma_\delta} \frac{\phi(s,z)( w - \zeta)^k}{(s-\zeta)^{k+1}}\,ds\bullet z - \int_{\varphi_\delta} \frac{\phi(s,z)}{s- w}\,ds\bullet z\Big{)} = 0 \\
\]

Of which these individual contour integrals are similar to the general class $\oint$. Though really only in the limit the equality is had; it is obviously much less symbol heavy to use the congruent integral, and make the appropriate conjugations to return to the compositional integral scenario. Additionally, the convergence is guaranteed using the congruent integral; it is not in the pull back to the compositional integral. 

Though, we should say a more direct explanation of this formula. Note to the reader, it isn't pretty.  Let us denote ${\displaystyle f_k(s,z) = \frac{\phi(s,z)(w-\zeta)^k}{(s-\zeta)^{k+1}}}$; if only to subtract the amount of symbols in the slightest. There are a sequence of arcs $\tau^{n}_{jk}$ and $\mu^{n}_{l}$ (which are not closed) such that,

\begin{eqnarray*}
&&\lim_{n\to\infty}\OmSum_l \int_{\mu^n_l}\bullet \OmSum_{k=0}^\infty \Big{(}\OmSum_j \int_{\tau^{n}_{jk}} \bullet \int_\gamma f_k(s,z)\,ds\bullet \MhSum_j \int_{(\tau^{n}_{jk})^{-1}} \Big{)}\bullet \MhSum_l \int_{(\mu^{n}_l)^{-1}}\bullet z =\\
&=&\int_\gamma \frac{\phi(s,z)}{s-w}\,ds\bullet z\\
\end{eqnarray*}

Where each integral $\int_{\tau_{jk}^n}$ is defined by some sequence of meromorphic functions $t_{jk}^n(s) : \mathcal{S} \to \widehat{\mathbb{C}}$. Such that,

\[
\int_{\tau_{jk}^n} = \int_{\tau_{jk}^n} t_{jk}^n(s)\phi(s,z)\,ds\bullet z\\
\]

And similarly for $\mu_{l}^n$. Obviously it makes sense to shorten this whole mess of equations with the symbol $\oint$. Next to this of course, is just how much more aesthetically pleasing it looks in the congruent integral scenario. 
Nicely enough; in the congruent integral scenario: $\oint$, $\OmSum$, and $ds\bullet z$ behave isomorphically to $\int$, $\sum$ and $ds$. This is the take-away to be made from this chapter.

\section*{The semi-group property of residuals}\label{secAT5}
\setcounter{equation}{0}

We will take a quick break here, before moving on into murkier waters. The author felt it important to point out a consequence of The Additivity Theorem \ref{thmADD}. Suppose that $\mathcal{S}\subset\mathbb{C}$ is a simply connected domain. Let $\phi(s,z): \mathcal{S} \times \mathbb{C} \to \mathbb{C}$ be a holomorphic function. Let $f(s) : \mathcal{S} \to \widehat{\mathbb{C}}$ be a meromorphic function. Let $\gamma:[a,b] \to \mathcal{S}$ be a Jordan curve. 

We're going to express a semi-group property inherent to congruent integration. This serves to describe a very complex creature. It looks innocent enough; but behaves abnormally. Let $F( w,z)$ be defined as,

\[
F( w,z) = \oint_\gamma  w f(s)\phi(s,z)\,ds\bullet z
\]

Then, recalling we're in the reduced group $\mathcal{A} = \mathcal{P} / \simeq$,

\[
F( w,F(\alpha,z)) = F( w + \alpha,z)\\
\]

It helps now, to be a tad more explicit of the reduced group $\mathcal{A}$. We factually need to remain in the group $\mathcal{A}$ of coherent integrals, but we can take pull-backs to the composition integral. The above ``semi-group" composition can be better written,

\[
\widetilde{F}(w,z) = \int_\gamma  w f(s)\phi(s,z)\,ds\bullet z
\]

So that,

\[
g \bullet \widetilde{F}(w) \bullet g^{-1} \bullet h \bullet \widetilde{F}(\alpha) \bullet h^{-1} = r \bullet \widetilde{F}(w + \alpha) \bullet r^{-1}\\
\]

Where $g,h,r : \mathcal{D} \to \mathbb{C}$ are holomorphic functions such that $\overline{\mathcal{D}} = \mathbb{C}$. So when we write things like the above semi-group formula. This is what is intended. But, if we take a limit in $\mathcal{A}$ we must walk a bit carefully when taking the pullback to the compositional integral.

Limits in $\mathcal{A}$ behave a tad differently than limits when focusing on the compositional integral. As this chapter has explored. For instance, if we let $\gamma_\delta : [0,2\pi] \to \mathcal{S}$ and $x \mapsto \zeta+ \delta e^{ix}$, then in $\mathcal{A}$,

\[
\int_{\gamma_\delta} \in \oint_{\gamma_\delta}
\]

But,

\[
\lim_{\delta \to 0}\int_{\gamma_\delta}\,\,\text{does not necessarily converge}\\
\]

While,

\[
\lim_{\delta \to 0}\oint_{\gamma_\delta} = \Rsd\\
\]

Because each,

\[
\oint_{\gamma_\delta} = \Rsd\\
\]

Which is where the magic really happens. The fact we can limit the compositional side while we keep the congruent side constant. What we \emph{can} do in the pull back is a bit different. Suppose that $\varphi_\delta$ is some other Jordan curve about $\zeta$ which tends to zero as $\delta$ shrinks to zero, then,

\[
\lim_{\delta \to 0}\int_{\gamma_\delta} - \lim_{\delta \to 0}\int_{\varphi_\delta} \,\,\text{does converge}\\
\]

Which follows because they are conjugate similar, and the conjugate shrinks; and since the conjugate tends to zero, so does the difference. So additionally, when we consider the congruent integral; we should keep in mind when doing the pullback to the compositional integral; we require that we derive equality between two things as there being a limit which satisfies:

\[
\lim_{n\to\infty} f_n - h_n = 0\\
\]

So when we write equality in the reduced group $\mathcal{A}$; if the result is discovered through some limit in $\mathcal{A}$ then we need to take the limit in the above manner. Where $f_n, h_n\to\infty$ but $f_n - h_n \to 0$ and so we define $\lim_{n\to\infty} f_n = \lim_{n\to\infty} h_n$.

So although,

\[
F(w) \bullet F(\alpha)\bullet z = F(w+\alpha)
\]

is an easy consequence of The Additivity Theorem \ref{thmADD}; it serves to complicate matters greatly in complex dynamics and the pull back. It shows the anomaly of the residual and the congruent integral. It derives a homomorphism from addition to composition. The function $F$ is no less than a semi-group; but it's a bit of a jargon salad. It includes integration and conjugation in a very estranged manner.

\chapter{\emph{semi}Semi-groups, and discussions of Fourier}\label{chap4}

\section{Introduction}\label{sec1}
\setcounter{equation}{0}

The purpose of this chapter is to provide a more thorough analysis of the semi-group structure of the congruent integral. And, furthermore to talk about taking strange integrals across these semi-groups. To set our goals in stone, is of no real direction. To set a theme, is a tad more understandable. That being, semi-groups, semi-groups, and more semi-groups; but really, they aren't semi-groups they're almost semi-groups, or \emph{semi}semi-groups. With an intention of creating what looks like an integral transform--specifically those of a Fourier variety, this chapter is written.

To the readers initial sensation of the subject we write the central object of study--of which we intend to advance. Let $f(s) : \mathcal{S}\to \widehat{\mathbb{C}}$ be a meromorphic function. Let $\phi(s,z) : \mathcal{S} \times \mathbb{C} \to \mathbb{C}$ be a holomorphic function. Of it, allow $\gamma : [a,b] \to \mathcal{S}$ to be a Jordan curve in the complex plane and let $ w \in \mathbb{C}$. Then,

\[
F( w,z) = \oint_\gamma  w f(s)\phi(s,z)\,ds\bullet z\\
\]

$F$ is a very special function. It's kind of like a semi-group, as previously observed. Meaning, more clearly that,

\[
F( w,F(\alpha,z)) = F( w + \alpha,z)\\
\]

Which for our purposes we'll write using the bullet notation to further express commutativity and its relation to the exponential.

\[
F( w) \bullet F(\alpha) \bullet z = F( w+\alpha)\\
\]

In no other place in this paper does the bullet notation appear more natural. It has completely replaced $\cdot$ (multiplication) in the exponential formula. This comparison works in an unorthodox way, though. We're doing a bunch of hidden conjugations in this mess. Each term is conjugated by an appropriate function.

Now as a sweeping goal, we want to excavate the relationship between $F$ and $f$'s residues. In fact $F$ is just a special kind of residual function. And it has a good strong footing in traditional analysis which will be unearthed as we progress.

But, it has a more difficult and rough idea at its core. The Additivity Theorem \ref{thmADD} and The Summation Theorem \ref{thmSUM} expressed a discrete interchange between composition and addition; but with integrals, the problem looms as a more difficult thing. How do we handle conjugations when integrating? We took a sum and turned it into composition, minus some conjugations. How can we take an integral and turn it into a compositional integral; all while applied to a closed contour, all while keeping track of our conjugations?

It's not really as easy as the notation will make it appear. It follies a fair amount of technical build up. It may appear as though the author will begin to add and add and add; without harnessing anything significant towards this goal. But it is done in this manner so that we may top everything off with some Fourier analysis; and it appears in an illuminating manner; as well as being an easy consequence by that point.

What'll be done in this chapter, most especially, is compositional integration across multiple variables; and resorting to the congruent integral to straighten it out. We will consider multiple integrals, rather than a singular integral across one variable. In doing this we broach a fair amount of uncharted waters and strange expressions.

This last chapter of this paper poses possibilities and the immense difficulty which awaits. We are engorging ourselves on the homomorphic nature of the congruent integral. We are looking to take the identity $\oint(f+g)\phi = \oint f\phi \bullet \oint g\phi$ to its extreme. And in doing so, describe how an interchange of integrals is possible.

\section{The Poor-Man's Residue Theorem}\label{sec2}
\setcounter{equation}{0}

In this section the following schema is used. Let $\mathcal{S}$ be a simply connected domain in $\mathbb{C}$. We will let $\phi(s,z) : \mathcal{S} \times \mathbb{C} \to \mathbb{C}$ be a holomorphic function. We let $\zeta \in \mathcal{S}$ and as well, we let $\gamma$ be a Jordan curve encircling $\zeta$ positively. The arcs $\tau_k\subset \mathcal{S}$ are a sequence of arcs that do not enclose $\zeta$; and $t_k : \mathcal{S} \to \widehat{\mathbb{C}}$ are a sequence of meromorphic functions. Further, $ w \in \mathbb{C}$. Then, the objects of interest are two things, for $j\in\mathbb{N}$,

\[
G( w,z) = \int_\gamma \frac{ w\phi(s,z)}{(s-\zeta)^{j+1}}\,ds\bullet z\\
\]

And,

\[
g_k( w,z) = \int_{\tau_k}  wt_k(s)\phi(s,z)\,ds\bullet z\\
\]

These expressions are holomorphic for $w \in \mathbb{C}$ and $z \in \mathcal{D} = \mathcal{D}(w)$ where $\overline{\mathcal{D}} = \mathbb{C}$. Now in the additive case, the first expression can be handled by Cauchy's Integral Formula, and Cauchy's Derivative Formula. We aren't so lucky in the compositional case. We are only left with a poor-man's version. For a sanity check, assume $\phi(s,z) = \phi(s)$ is constant in $z$, then,

\[
G( w,z) = z + \frac{2 \pi i  w}{j!} \phi^{(j)}(\zeta)\\
\]

And certainly this is an additive semi-group. No real surprise there; it's nothing more than a one-off. Letting $\phi(s,z) = p(s)g(z)$ be separable, a similar form arises.

\[
G(w,z)=\int_0^{2\pi iw}\frac{\phi^{(j)}(\zeta,z)}{j!}\,dx\bullet z\\
\]

Returning to the case where $\phi(s,z)$ is not separable; we don't get a closed form expression as nice. Instead, we get something a bit more symbol heavy. Remembering, in the general case, we have to account for a whole bunch of conjugations.

Now since we are trying to induce a semi-group through the congruent integral, we have to act a bit different. But what we want is,

\[
\frac{d G}{d  w}\Big{|}_{w=0} = \frac{2\pi i }{j!}\frac{\partial^j}{\partial \zeta^j}\phi(\zeta,z)= \int_\gamma \frac{\phi(s,z)}{(s-\zeta)^{j+1}}\,ds\\
\]

And,

\[
\frac{d g_k}{d  w}\Big{|}_{w=0} = \int_{\tau_k} t_k(s)\phi(s,z)\,ds\\
\]

To write it in the manner we wish to exploit it as is a bit more cryptic. It may not ring as evident in this form; but that's what this chapter is for; to remove the encryption. Nonetheless,

\[
G( w,z) = z +  w \int_\gamma \frac{\phi(s,z)}{(s-\zeta)^{j+1}}\,ds + \mathcal{O}( w^2)\\
\] 

\[
g_k(w,z) = z + w \int_{\tau_k} t_k(s)\phi(s,z)\,ds + \mathcal{O}(w^2)\\
\]

Note that the first order term is not a compositional integral; it's just your ordinary run of the mill additive kind. This is to show, that in a neighborhood of $0$, $G_j( w,z) \sim z +  w\int_\gamma \frac{\phi(s,z)}{(s-\zeta)^{j+1}}\,ds$ looks just like a run of the mill integral.

This will become our equation of a semi-group. But, for the congruent integral. If we write,

\[
\widetilde{G}(w,z) = \oint_\gamma \frac{w\phi(s,z)}{(s-\zeta)^{j+1}}\,ds\bullet z\\
\]

Where,

\[
\widetilde{G} = \OmSum_k g_k \bullet G 
\bullet \MhSum_k g_k^{-1} \bullet z\\
\]

Now, the above use of first order approximation allows us to say that,

\[
\frac{d \widetilde{G}}{dw}\Big{|}_{w=0} = \frac{dG}{dw}\Big{|}_{w=0}\\
\]

Because,

\begin{eqnarray*}
\widetilde{G} &=& \OmSum_k g_k \bullet G \bullet \MhSum_k g_k^{-1}\bullet z\\
&=& \OmSum_k \big{(}z + w\int_{\tau_k}t_k(s)\phi(s,z) \,ds\big{)} \bullet G \bullet \MhSum_k \big{(}z - w\int_{\tau_k}t_k(s)\phi(s,z) \,ds\big{)}\bullet z + \mathcal{O}(w^2)\\
&=& z+ \sum_k w\int_{\tau_k}t_k(s)\phi(s,z) \,ds + G - \sum_k w\int_{\tau_k}t_k(s)\phi(s,z) \,ds + \mathcal{O}(w^2)\\
&=& z + G + \mathcal{O}(w^2)\\
\end{eqnarray*}

This is all derived through the first order approximation. But it allows us to again, denote the congruent integral in a parallel to the compositional integral. This is to say,

\[
\frac{d}{dw}\Big{|}_{w=0}\oint_\gamma \frac{w\phi(s,z)}{(s-\zeta)^{j+1}}\,ds\bullet z = \int_\gamma \frac{\phi(s,z)}{(s-\zeta)^{j+1}}\,ds = \frac{d}{dw}\Big{|}_{w=0}\int_\gamma \frac{w\phi(s,z)}{(s-\zeta)^{j+1}}\,ds\bullet z\\
\]

So, all in order there. Of this we now must prove the base assumption of these derivations. Of which we write the perhaps most important lemma of this chapter.

\begin{lemma}[The Infinitesimal Lemma]\label{lmaInf}
Suppose $\mathcal{S}$ is a simply connected domain. Suppose $\phi(s,z):\mathcal{S}\times \mathbb{C} \to \mathbb{C}$ is a holomorphic function. Suppose $f(s):\mathcal{S} \to \widehat{\mathbb{C}}$ is a meromorphic function. Suppose $w \in \mathbb{C}$. Then for an arbitrary arc $\sigma \subset \mathcal{S}$,

\[
\frac{d}{dw}\Big{|}_{w=0}\int_\sigma wf(s)\phi(s,z)\,ds\bullet z = \int_\sigma f(s)\phi(s,z)\,ds\\
\]

\end{lemma}

\begin{proof}
Take the partial summations of this integral and,

\[
\int_\sigma wf(s)\phi(s,z)\,ds\bullet z = z + w\sum_{m} f(s_m) \phi(s_m,z_m(w)) \Delta s_m\\
\]

Now $z_m(w) \to z$ as $w\to0$. Let $\Delta s_m \to 0$ to get the integral, which concludes the proof.
\end{proof}

From this lemma we derive a theorem:

\begin{theorem}[The Poor-man's Residue Theorem]\label{thmPRST}
Suppose $\mathcal{S}$ is a simply connected domain. Suppose $\phi(s,z):\mathcal{S}\times \mathbb{C} \to \mathbb{C}$ is a holomorphic function. Suppose $f(s):\mathcal{S} \to \widehat{\mathbb{C}}$ is a meromorphic function. Let $\gamma$ be a Jordan curve. Then,

\[
\frac{d}{dw}\Big{|}_{w=0} \oint_{\gamma} wf(s)\phi(s,z)\,ds\bullet z = \int_\gamma f(s)\phi(s,z)\,ds\\
\]
\end{theorem}

So despite all the conjugations the result in the infinitesimal is as expected. The congruent integral is similar to the compositional integral.

\section{Integrating in The Separable Case}\label{sec4}
\setcounter{equation}{0}

In this section the following schema is used. Let $\mathcal{S}$ be a simply connected domain in $\mathbb{C}$. We will let $f(s) : \mathcal{S} \to \widehat{\mathbb{C}}$ be a meromorphic function. Let $p( w) : \mathbb{C} \to \mathbb{C}$ be a holomorphic function. We will also let $\phi(s,z) : \mathcal{S} \times \mathbb{C} \to \mathbb{C}$ be a holomorphic function. Assume that $\gamma:[a,b] \to \mathcal{S}$ is a Jordan curve oriented positively.

As before, we'll conform to the restriction that,

\[
F( w,z) = \oint_\gamma  w f(s)\phi(s,z)\,ds\bullet z\\
\]

However, the object we'll be interested in is,

\[
F(p( w),z) = \oint_\gamma p( w)f(s)\phi(s,z)\,ds\bullet z\\
\]

But before we hit this, we'll look more closely at $F( w)$ without $p$. When we take the derivative in $w$ we are doing something which may be off kilter. But nonetheless it can on occasion be helpful to go off script. We can look at this like the derivative of an exponential. That is to say,

\[
\lim_{\delta \to 0} \frac{F( w+\delta,z) - F( w,z)}{\delta} = \lim_{\delta \to 0} \frac{F(\delta,z) - z}{\delta} \bullet F \bullet z\\
\]

Which is to say, very little different than the exponential; and the reduction that ${\displaystyle \frac{e^{ w + \delta} - e^{ w}}{\delta} = \frac{e^{\delta} - 1}{\delta} \cdot e^{ w}}$. We can in essence though, reduce all of $F$'s relation to $f$ and $\phi$ to the value,

\[
\lim_{\delta \to 0} \frac{F(\delta,z) - z}{\delta} = 2\pi i \sum_j \Res_{s=\zeta_j} f(s)\phi(s,z)\\
\]

Here, to be clear, the limit $\delta \to 0$ is taken in the complex sense. From all paths approaching $0$. This is the first order approximation we are so keen on.

This notation can be made more clear by the identity,

\[
\frac{d}{d w}\Big{|}_{ w = 0} F( w,z) = 2\pi i \sum_j \Res_{s=\zeta_j} f(s)\phi(s,z)\\
\]

However, an even clearer form of this identity, and the true form we want:

\[
\frac{d}{d w}\Big{|}_{ w = 0} F( w,z) = \int_\gamma f(s)\phi(s,z)\,ds\\
\]

In his internal language the author likes to refer to this as the \textit{logarithm} of the semi-group. This comes from its direct relation to the exponential case. Which is to say that,

\[
\log(F;z) = \frac{d}{d w}\Big{|}_{ w = 0} F( w,z)
\]

If $E( w,z) = ze^{\lambda  w}$ is the exponential semi-group for some $\lambda \in \mathbb{C}$, then the logarithm of this reduces to our usual logarithm more or less,

\[
\log(E;z) = \lambda z
\]

Where we note a multiplicative factor. More egregiously, we still have a nice additive identity that the logarithm version is contained within. In that if,

\begin{eqnarray*}
F_1 &=& \oint_\gamma  w f_1(s)\phi(s,z)\,ds\bullet z\\
F_2 &=& \oint_\gamma  w f_2(s)\phi(s,z)\,ds\bullet z\\
F_1 \bullet F_2 = H &=& \oint_\gamma  w \big{(}f_1(s) + f_2(s)\big{)}\phi(s,z)\,ds\bullet z\\
\end{eqnarray*}

Then,

\[
\log(H;z) = \log(F_1;z) + \log(F_2;z)\\
\]

Or,

\[
\log(F_1 \bullet F_2;z) = \log(F_1;z) + \log(F_2;z)\\
\]

We should always recall that $\log = \log_\phi$. This logarithm function is defined with respect to $\phi(s,z)$. But forgoing this; we assume we have always fixed $\phi$.

The relation to the logarithm should be glaringly obvious. Much of the remainder of this paper hinges entirely on this discussion of logarithms; but the author is wary to include it. It does aid in visualizing what exactly is going on in the residue theorems of the previous sections; however, it is a precarious thing. It is difficult to put on solid ground.\\

Now returning to the discussion above, we enter in the more important function,

\[
F(p( w),z)\\
\]

And in that manner, we can familiarize ourselves a bit with its structure. The most important factor being,

\[
F(p_1( w) + p_2( w)) = F(p_1( w))\bullet F(p_2( w)) \bullet z\\
\]

Now to set the stage, we're going to integrate this object. But not in a traditional sense. In the compositional integral sense. Let $\tau$ be some continuously differentiable arc in $\mathbb{C}$. Observe the following manipulation of Riemann Sums,

\begin{eqnarray*}
F(\int_\tau p( w)\,d w,z) &=& F(\lim_{\Delta \tau_j \to 0} \sum_{j=0}^{n-1} p(\tau_j^*)\Delta \tau_j,z)\\
&=& \lim_{\Delta \tau_j \to 0} F( \sum_{j=0}^{n-1} p(\tau_j^*)\Delta \tau_j,z)\\
&=& \lim_{\Delta \tau_j \to 0} \OmSum_{j=0}^{n-1} F(p(\tau_j^*)\Delta \tau_j,z)\bullet z\\
&=& \lim_{\Delta \tau_j \to 0} \OmSum_{j=0}^{n-1} z + p(\tau_j^*) \log(F;z)\Delta \tau_j \bullet z\\
&=& \int_\tau p( w) \log(F;z)\,d w \bullet z\\
\end{eqnarray*}

Where the third to fourth line was derived by first order approximation. It helps now to slightly abandon this logarithm notation and recall its expansion. Which reduces to,

\[
\oint_\gamma \Big{(}\int_\tau p( w)\,d w\Big{)} f(s)\phi(s,z)\,ds\bullet z= \int_\tau 2 \pi i p( w)\sum_j \Res_{s=\zeta_j} f(s)\phi(s,z)\,d w \bullet z\\
\]

And from here, we can see a more clear picture. If we call $h( w,s) = p( w)f(s)$ then,

\[
\oint_\gamma \Big{(}\int_\tau h( w,s)\,d w\Big{)}\phi(s,z)\,ds\bullet z= \int_\tau 2 \pi i\sum_j \Res_{s=\zeta_j} h( w,s)\phi(s,z)\,d w \bullet z\\
\]

Even more miraculously, we can write this in a language not too different from Fubini. Which is to say, we can write this as the interchange of integrals. 

\[
\oint_\gamma \int_\tau h( w,s)\phi(s,z)\,d w\,ds\bullet z = \int_\tau \int_\gamma h( w,s)\phi(s,z)\,ds\,d w\bullet z\\
\]

Deriving this for the case when $h( w,s)$ is not separable will be the goal of the next section. From that; discussions of Fourier's analysis extended to the compositional scenario appear a tad more naturally.

\section{The Poor Man's Fubini's Theorem}\label{sec5}
\setcounter{equation}{0}

Throughout this section we let $\mathcal{S}$ be a simply connected domain in $\mathbb{C}$ and let $\mathcal{W}$ be a domain. We will assume that $\phi(s,z) : \mathcal{S} \times \mathbb{C} \to \mathbb{C}$ is a holomorphic function. We will assume that $h( w,s) : \mathcal{W} \times \mathcal{S} \to \widehat{\mathbb{C}}$ is a meromorphic function. Let $\gamma$ be a Jordan curve in $\mathcal{S}$; and let $\tau$ be a continuously differentiable arc in $\mathcal{W}$.

The object of this section is to derive something that looks like Fubini's theorem. This is to mean, in our current zoology; some kind of rule that gives us the same legwork Fubini gave the traditional functional analysis. It masks itself as an interchange of integrals; but it's a bit more subtle than that. It's just as necessarily a statement that $d w\,ds\bullet z = ds\,d w\bullet z$; if the conditions are just right. This means differentials commute when attached to $z$ with a $\bullet$-product.

But what we want is,

\[
\oint_\gamma \int_\tau h( w,s)\phi(s,z)\,d w\,ds\bullet z = \int_\tau \int_\gamma h( w,s)\phi(s,z)\,ds\,d w\bullet z\\
\]

To begin we can take a calculated approach and work with a Riemann-Stieljtes Composition again. In hopes of aquiring brevity; with a loss of symbols; the author hopes the reader forgives him if he just writes,

\begin{eqnarray*}
\oint_\gamma \int_\tau h( w,s)\phi(s,z)\,d w\,ds\bullet z &=& \oint_\gamma \lim_{\Delta \tau_j \to 0}\sum_j h(\tau_j^*,s)\phi(s,z) \Delta \tau_j \, ds\bullet z\\
&=& \lim_{\Delta \tau_j \to 0} \OmSum_j \oint_\gamma h(\tau_j^*,s)\phi(s,z)\,\Delta \tau_j\,ds\bullet z\\
\end{eqnarray*}

But,

\[
\oint_\gamma h(\tau_j^*,s)\phi(s,z)\,\Delta \tau_j\,ds\bullet z = z + \int_\gamma h(\tau_j^*,s)\phi(s,z)\,ds \Delta \tau_j + \mathcal{O}(\Delta \tau_j^2)\\
\]

Which follows from the first order expansions done in The Infinitesimal Lemma \ref{lmaInf} and The Poor Man's Residue Theorem \ref{thmPRST}. And with that we arrive at,

\[
\oint_\gamma \int_\tau h( w,s)\phi(s,z)\,d w\,ds\bullet z = \lim_{\Delta \tau_j \to 0} \OmSum_j z + \int_\gamma h(\tau_j^*,s)\phi(s,z)\,ds \Delta \tau_j\bullet z\\
\]

Which in the limit certainly converges to our desired result. Giving us a taped together Fubini's Theorem. Which is to say, a poor man's Fubini's Theorem.

\begin{theorem}[The Poor Man's Fubini's Theorem]\label{thmFUB}
Let $\mathcal{S}$ be a simply connected domain in $\mathbb{C}$ and let $\mathcal{W}$ be a domain. Let $\phi(s,z) : \mathcal{S} \times \mathbb{C} \to \mathbb{C}$ be a holomorphic function and let $h( w,s) : \mathcal{W} \times \mathcal{S} \to \widehat{\mathbb{C}}$ be a meromorphic function. Let $\gamma$ be a Jordan curve in $\mathcal{S}$; and let $\tau$ be a continuously differentiable arc in $\mathcal{W}$. Then,

\[
\oint_\gamma \int_\tau h( w,s)\phi(s,z)\,d w\,ds\bullet z = \int_\tau \int_\gamma h( w,s)\phi(s,z)\,ds\,d w\bullet z\\
\]
\end{theorem}

\section{A Primer of Fourier}\label{sec6}
\setcounter{equation}{0}

Throughout this section we let $\mathcal{S}$ be a simply connected domain in $\mathbb{C}$; and let $\mathcal{W}_a = \{ w \in \mathbb{C} :\,|\Im( w)| < a\}$ for some $a>0$. We will assume that $\phi(s,z) : \mathcal{S} \times \mathbb{C} \to \mathbb{C}$ is a holomorphic function. We will assume that $h( w,s) : \mathcal{W}_a \times \mathcal{S} \to \widehat{\mathbb{C}}$ is a meromorphic function. Let $\gamma$ be a Jordan curve in $\mathcal{S}$. For convenience $h(w,s)$ is holomorphic in $w$ for all $s \in \gamma$. One can think of this as the singularities in $h$ being isolated in $s$; though not necessarily--but for convenience assume so.

The object of study in this section is a bivariable function $f$ and its strange Fourier pair $\widehat{f}$. To begin, we write and define,

\[
f( w,z) = \oint_\gamma h( w,s)\phi(s,z)\,ds \bullet z\\
\]

Written this way, we call this a \emph{Derived Residual}. Derived Residuals will make up the rest of this paper. It's not quite a semi-group; but it shares enough of its properties to be useful. Of it, we make use.

With this description, we want to create a Fourier Transform. It's not immediately obvious how, but we would like the corresponding transform to be,

\[
\widehat{f}(\xi,z) = \oint_\gamma \widehat{h}(\xi,s)\phi(s,z)\,ds \bullet z\\
\]

Where $\widehat{h}$ is the usual Fourier Transform of $h$ across $w$. Which is to say,

\[
\widehat{h}(\xi,s) = \int_{-\infty}^{\infty} h(w,s)e^{- 2 \pi i w\xi}\,dw\\
\]

Which now we can see that it is incredibly useful for the singularities of $h$ to be isolated in $s$. For a sanity check; we can typically assume this is the case. We have quite a few road blocks ahead of us to define and control this function. The most stringent being; what operation do we actually perform on $f$ to get $\widehat{f}$?

Well if we go with the convention that,

\[
\log(f;w,z)= \int_\gamma h( w,s)\phi(s,z)\,ds
\]

Which follows closely to our work done in these Poor Man's Theorems. We can almost see what we need. Except, we don't need this precisely. We need something close to it. And in that vein, we attempt to clarify the logarithm function. What we want to use is the following,

\[
\frac{d}{d\epsilon}\Big{|}_{\epsilon = 0} \int_\gamma \epsilon h( w,s)\phi(s,z)\,ds \bullet z = \int_\gamma h( w,s)\phi(s,z)\,ds\\ 
\]

For which we can prescribe a kind of functional derivative,

\[
\log(f;  w,z): \int_\gamma h( w,s)\phi(s,z)\,ds \bullet z \mapsto \int_\gamma h( w,s)\phi(s,z)\,ds\\
\]

Wherein, this logarithm function essentially just deletes the $\bullet z$. Obviously this requires care, but foregoing this transgression momentarily, we continue. By The Poor Man's Fubini's Theorem \ref{thmFUB}:

\begin{eqnarray*}
\widehat{f}(\xi,z) &=& \int_{-\infty}^\infty e^{-2 \pi i \xi  w} \log(f; w,z)\,d w\bullet z\\
&=& \int_{-\infty}^\infty e^{-2 \pi i \xi  w} \int_\gamma h( w,s)\phi(s,z)\,ds\,d w\bullet z\\
&=& \oint_\gamma \int_{-\infty}^\infty h( w,s)e^{-2 \pi i \xi  w}d w \, \phi(s,z)\,ds\bullet z\\
&=& \oint_\gamma \widehat{h}(\xi,s)\phi(s,z)\,ds \bullet z\\
\end{eqnarray*}

This we will take as the definition of our integral transform. In a very similar form the inverse transform looks like,

\[
f( w,z) = \int_{-\infty}^\infty e^{2 \pi i \xi  w} \log(\widehat{f};\xi,z)\,d\xi\bullet z\\
\]

Sadly this is only a primer. In order for all these progressions to be factual; we need this logarithm function to be well defined. This is the hardest part of defining The Compositional Fourier Transform. Actually it's pretty straight-forward, but it may take a small perturbation of just how malleable these expressions behave; and what they mean clearly.

Assume we have some function $f$, such that $h_1 \neq h_2$,

\[
f( w,z) = \int_\gamma h_1( w,s)\phi(s,z)\,ds \bullet z = \int_\gamma h_2( w,s)\phi(s,z)\,ds \bullet z\\
\]

And,

\[
\int_\gamma h_1( w,s)\phi(s,z)\,ds  \neq\int_\gamma h_2( w,s)\phi(s,z)\,ds\\
\]

We need to manage the cases in which this happens. As to this, we move towards the next section; and attempt to put the logarithm on strong footing.

\section{Reifying The Logarithm Function}\label{sec7}
\setcounter{equation}{0}

Throughout this section we let $\mathcal{S}$ be a simply connected domain in $\mathbb{C}$; and let $\mathcal{W}_a = \{ w \in \mathbb{C} :\,|\Im( w)| < a\}$ for some $a>0$. We will assume that $\phi(s,z) : \mathcal{S} \times \mathbb{C} \to \mathbb{C}$ is a holomorphic function. We will assume that $h( w,s) : \mathcal{W}_a \times \mathcal{S} \to \widehat{\mathbb{C}}$ is a meromorphic function. Let $\gamma$ be a Jordan curve in $\mathcal{S}$. We let,

\[
f[h]( w,z) = \int_\gamma h( w,s)\phi(s,z)\,ds\bullet z\\
\]

To make sense of the logarithm; we have to first concede that as we've talked about it so far is ill-defined. There can be no logarithm function defined on $f$; instead it needs to be defined on $f[h]$. So to that end, we need to change up our language a fair amount.

Now as noted, the quick-way to define the logarithm function is,

\[
\lim_{\epsilon \to 0} \frac{f[\epsilon h] - z}{\epsilon} = \log(f[h]; w,z)\\
\]

Which is what we will qualify as our logarithm. The reason this is not an operation on $f$ is because if it were, we'd have multiply defined values. For the sake of the argument, let $\phi(s,z) = z$ and take,

\[
f[\frac{p( w)}{s-\zeta}]( w,z) = \int_{|s| = 1}\frac{p( w)}{s-\zeta} z\,ds\bullet z = z e^{2 \pi i p( w)}\\
\]

Then, the $\log$ of this thing is just,

\[
\log(f[\frac{p( w)}{s-\zeta}]; w,z) = 2 \pi i p( w)z\\
\]

Now let's observe another trick,

\begin{eqnarray*}
f[\frac{p( w) + 1}{s-\zeta}]( w,z) &=& \int_{|s| = 1}\frac{p( w) + 1}{s-\zeta} z\,ds\bullet z\\
&=& z e^{2 \pi i p( w) + 2\pi i}\\
&=& z e^{2 \pi i p( w)}\\
&=& f[\frac{p( w)}{s-\zeta}]; w,z)
\end{eqnarray*}

However, if we were to take the logarithm function we'd get,

\[
\log (f[\frac{p( w) + 1}{s-\zeta}] ;  w,z) = 2 \pi i (p( w)+1)z\\
\]

So as to entice the reader; the log function is defined on $f[h]$; it is not defined on $f$. As two $h$'s may produce the same $f$, but they induce different logarithms. This really doesn't pose much of a problem, but let's walk through a case where it possibly could.

Assume that $p( w)$ is a Fourier Transformable function; i.e: that it is $L^1$. Then,

\[
f( w,z) = ze^{\displaystyle 2 \pi i p( w)}\\
\]

And The Fourier Transform seems to be two different things,

\[
\widehat{f}(\xi,z) = \int_{-\infty}^\infty e^{-2 \pi i \xi  w}2 \pi i p( w)z\,d w \bullet z = ze^{\displaystyle 2 \pi i \widehat{p}(\xi)}
\]

And then secondly this,

\[
\widehat{f}(\xi,z) = \int_{-\infty}^\infty e^{-2 \pi i \xi  w}2 \pi i (p( w)+1)z\,d w \bullet z = \infty\\
\]

But, if we were clear about this operation, we would say that the Fourier operator $\widehat{\cdot}$ is actually performed across $h_1 = \frac{p( w)}{s-\zeta}$ and not performed across $h_2 = \frac{p( w)+1}{s-\zeta}$. One easy way to filter this out, is that $h_2$ is not $L^1$. The Fourier Transform makes no sense on $h_2$.

To this end; the logarithm function is defined across the function space $f[h]$ and not defined across the function space $f$. Furthermore, we recall that the logarithm function must be defined with respect to the function $\phi$.

All in all, there isn't much of a problem. It's just necessary that we're clear when we take The Compositional Fourier Transform; it's technically taken across $f[h]$ not $f$ itself. And that necessarily it depends on $h$. And in the spirit of this notation, we can say that,

\[
\widehat{f[h]} = f[\widehat{h}]\\
\]

With this straightened up, the author will still use the abuse of notation $\log(f; w,z)$ so long as the context is clear. Though it is to be remembered that this function depends on $h$; and this is an operation on $h$ and not $f$. And furthermore, $h$ is a variable in and of itself; existing in the function space $h:\mathcal{W}_a \times \mathcal{S} \to \widehat{\mathbb{C}}$ mapping to the space $f[h]$ where $f[h_1 + h_2] = f[h_1]\bullet f[h_2]$.

\section{A Fresh Coat of Fourier}\label{sec8}
\setcounter{equation}{0}

Throughout this section we let $\mathcal{S}$ be a simply connected domain in $\mathbb{C}$; and let $\mathcal{W}_a = \{ w \in \mathbb{C} :\,|\Im( w)| < a\}$ for some $a>0$. We will assume that $\phi(s,z) : \mathcal{S} \times \mathbb{C} \to \mathbb{C}$ is a holomorphic function. Let $\gamma$ be a Jordan curve in $\mathcal{S}$. We will assume that $h( w,s) : \mathcal{W}_a \times \mathcal{S} \to \widehat{\mathbb{C}}$ is a meromorphic function, which the singularities of $h$ are isolated in $s$.  We let,

\[
f[h]( w,z) = \oint_\gamma h( w,s)\phi(s,z)\,ds\bullet z\\
\]

As to the what to do; although we have just argued for The Fourier Transform; we haven't bothered to properly construct it yet. And init, we have not shown that any of these expressions actually make sense; that they don't diverge to infinity. Or when in fact, it does converge.

This section shall center around the work of Stein \& Shakarchi and their analysis of The Fourier Transform in \cite{SteSha}. To that end, we begin by assuming that $h( w,s)$ is of moderate decay. This is to mean for all $s \in \gamma$ there exists some $A > 0$ such that,

\[
|h(u + iv,s)| \le \frac{A}{1 + u^2}\\
\]

In Stein \& Shakarchi, if $h$ is of moderate decay, then $\widehat{h}$ exists and it too is of moderate decay. Even better it has exponential decay of order $2 \pi a>0$; where $a$ is the width of the strip $\mathcal{W}_a$. This is to mean for some $M>0$,

\[
|\widehat{h}(\xi,s)| \le M e^{-2 \pi a |\xi|}\\
\]

We are given the clear statement, without confusion, that,

\[
\log(f; w,z) = \int_\gamma h( w,s)\phi(s,z)\,ds\\
\]

Now one can clearly surmise that for all compact sets $\mathcal{K} \subset \mathbb{C}$ there exists an $A_\mathcal{K}$ such that,

\[
||\log(f;u+iv,z)||_{\mathcal{K}} \le \frac{A_{\mathcal{K}}}{1 + u^2}\\
\]

So that $\log$ is of moderate decay. Now in essence we're going to abstract this statement for our next theorem. In truth, this theorem is the summation of most of the work that's left.

\begin{theorem}
Suppose $g(t,z) : \mathbb{R} \times \mathbb{C} \to \mathbb{C}$ is continuously differentiable in $t$ and holomorphic in $z$. Assume for all compact disks $\mathcal{K} \subset \mathbb{C}$ that,

\[
\int_{-\infty}^\infty ||g(t,z)||_{z \in \mathcal{K}}\,dt < \infty\\
\]

If,

\[
G(z) = \int_{-\infty}^\infty g(t,z)\,dt\bullet z\\
\]

Then $G$ is holomorphic for $z \in \mathcal{D}$ such that $\overline{\mathcal{D}} = \mathbb{C}$.
\end{theorem}

\begin{proof}
Take $T'>T>0$ and use The Triangle Inequality Theorem \ref{thmTriInq}. For all compact subsets $\mathcal{K} \subset \mathbb{C}$ there exists some $\mathcal{K} \subset \mathcal{J}$ and large enough $T$, such that,

\[
\Big{|}\Big{|}\int_{T}^{T'} g(t,z)\,dt\bullet z - z\Big{|}\Big{|}_{z \in \mathcal{K}} \le \int_{T}^{T'} ||g(t,z)||_{z \in \mathcal{J}}\,dt < \epsilon\\
\]

And similarly,

\[
\Big{|}\Big{|}\int_{-T'}^{-T} g(t,z)\,dt\bullet z - z\Big{|}\Big{|}_{z \in \mathcal{K}} < \epsilon\\
\]

Where $\epsilon$ is as small as desired by letting $T' > T > T_0$ be as large as possible. In this vein then,

\[
\int_{-\infty}^\infty = \int_T^\infty \bullet \int_{-T}^T \bullet \int_{-\infty}^{-T} = (z+\epsilon)\bullet \int_{-T}^T \bullet (z+\epsilon)\\
\]

If we can derive normality of the integral $\int_{-T}^T$, convergence follows suit. To derive normality, for $\mathcal{K} \subset \mathcal{H}$ a compact subset within a domain of holomorphy $\mathcal{H}$ for $T > T_0$; which is possible because the points of non-normality must be in some $\epsilon'$-neighborhood of each other--per what we've just shown. Using the triangle inequality,

\[
\Big{|}\Big{|}\int_{-T}^T g(t,z)\,dt\bullet z - z\Big{|}\Big{|}_{z \in \mathcal{K}} \le \int_{-T}^T ||g(t,z)||_{z \in \mathcal{J}}\,dt\\
\]

For $\mathcal{K} \subset \mathcal{J}$; of which as $T \to \infty$, the left hand side must remain bounded. Recalling that we are evoking compact normal convergence in $z$; and all these objects are holomorphic in $z$ minus a set of measure zero; one can see the uniform convergence in $z$ on compact subsets of their domains of holomorphy. For $T' > T > T_0$,

\begin{eqnarray*}
\Big{|}\Big{|} \int_{-T'}^{T'} - \int_{-T}^{T}\Big{|}\Big{|}_{z \in \mathcal{K}} &<& \Big{|}\Big{|} (z + \epsilon)\bullet \int_{-T}^{T}\bullet (z+\epsilon) - \int_{-T}^{T}\Big{|}\Big{|}_{z \in \mathcal{K}}\\
&<&\epsilon''\\
\end{eqnarray*}

For all $\epsilon'' > 0$ for large enough $T_0$. Which indeed shows that the integral does converge sufficiently well enough to satisfy the hypothesis.
\end{proof}

With this we can state Fourier's Inversion Theorem.

\begin{theorem}[Fourier's Inversion Theorem]\label{thmFurInv}
Let $\mathcal{S}$ be a simply connected domain in $\mathbb{C}$; and let $\mathcal{W}_a = \{ w \in \mathbb{C} :\,|\Im( w)| < a\}$ for some $a>0$. Let $\phi(s,z) : \mathcal{S} \times \mathbb{C} \to \mathbb{C}$ be a holomorphic function. We will assume that $h( w,s) : \mathcal{W}_a \times \mathcal{S} \to \widehat{\mathbb{C}}$ is a meromorphic function of moderate decay in $w$, and isolated singularities in $s$. Let $\gamma$ be a Jordan curve in $\mathcal{S}$. If,

\[
f = \oint_{\gamma}h( w,s)\phi(s,z)\,ds\bullet z\\
\]

And,

\[
\widehat{f} = \int_{-\infty}^\infty e^{-2 \pi i \xi  w}\log(f; w,z)\,d w \bullet z = \oint \widehat{h}(\xi,s)\phi(s,z)\,ds\bullet z\\
\]

Then,

\[
f = \int_{-\infty}^\infty e^{2 \pi i \xi  w}\log(\widehat{f};\xi,z)\,d\xi \bullet z\\
\]
\end{theorem}

It's important to remember this equivalence is across the reduced group of congruent integrals. This is to mean $\widehat{\widehat{f}} \simeq f$. This entire discussion of the congruent integral, everything is up to conjugation.

It occurs to the author it may help the reader to visualize some examples of Fourier pairs. Fourier pairs from the separable case make for the best intuition. It's a bit more mysterious in the general case; as we are describing equalities between classes of functions; per the congruent integral. But when all is good and separable,

\begin{eqnarray*}
f( w,z) = z + p( w) &\Leftrightarrow& \widehat{f}(\xi,z) = z + \widehat{p}(\xi)\\
f( w,z) = z e^{\displaystyle p( w)} &\Leftrightarrow& \widehat{f}(\xi,z) = z e^{\displaystyle \widehat{p}(\xi)}\\
f( w,z) = \frac{1}{\displaystyle \frac{1}{z} + p( w)} &\Leftrightarrow& \widehat{f}(\xi,z) = \frac{1}{\displaystyle \frac{1}{z} +\widehat{p}(\xi)}\\
f( w,z) = \frac{1}{\displaystyle\sqrt[{n-1}]{\frac{1}{z^{n-1}} + p( w)}} &\Leftrightarrow& \widehat{f}(\xi,z) = \frac{1}{\displaystyle\sqrt[n-1]{ \frac{1}{z^{n-1}} +\widehat{p}(\xi)}}\\
f(w,z) = \log(e^z + p(w)) &\Leftrightarrow& \widehat{f}(\xi,z) = \log(e^z + \widehat{p}(\xi))\\
\end{eqnarray*}

The next thing to derive from these results is Poisson's Summation formula. If we look at the above equations, certainly,

\[
\OmSum_{n=-\infty}^\infty f(n,z) \bullet z = \OmSum_{n=-\infty}^\infty \widehat{f}(n,z)\bullet z\\
\]

This being a rephrasing of Poisson's Summation formula--through The Additive Theorem \ref{thmADD}. But, this arises far more generally than for the facile way it holds in the above examples. But nonetheless, we urge ourselves to show,

\begin{eqnarray*}
\OmSum_{n=-\infty}^\infty f(n,z) \bullet z &=& \OmSum_{n=-\infty}^\infty \oint_\gamma h(n,s)\phi(s,z)\,ds\bullet z\\
&=& \oint_\gamma \sum_{n=-\infty}^\infty h(n,s) \phi(s,z)\,ds\bullet z\\
&=& \oint_\gamma \sum_{n=-\infty}^\infty \widehat{h}(n,s) \phi(s,z)\,ds\bullet z\\
&=& \OmSum_{n=-\infty}^\infty \oint_\gamma \widehat{h}(n,s)\phi(s,z)\,ds\bullet z\\
&=& \OmSum_{n=-\infty}^\infty \widehat{f}(n,z) \bullet z\\
\end{eqnarray*}

Which reduces into nothing but an interchange theorem which is obtuse, to say the least. One shouldn't dissuade the result though. We have built up a lot of material to obtain this result. It is largely a statement about semi-groups; but a semi-group in a reduced conjugate class. The author hopes he has conveyed this comparison well and true. As the penultimate theorem of this paper it is intended to express the abstract nature of the summa.

\begin{theorem}[Poisson's Composition Formula]
Let $\mathcal{S}$ be a simply connected domain in $\mathbb{C}$; and let $\mathcal{W}_a = \{ w \in \mathbb{C} :\,|\Im( w)| < a\}$ for some $a>0$. Let $\phi(s,z) : \mathcal{S} \times \mathbb{C} \to \mathbb{C}$ be a holomorphic function. We will assume that $h( w,s) : \mathcal{W}_a \times \mathcal{S} \to \widehat{\mathbb{C}}$ is a meromorphic function of moderate decay in $w$ with isolated singularities in $s$. Let $\gamma$ be a Jordan curve in $\mathcal{S}$. If,

\[
f = \oint_{\gamma}h( w,s)\phi(s,z)\,ds\bullet z\\
\]

And,

\[
\widehat{f} = \int_{-\infty}^\infty e^{-2 \pi i \xi  w}\log(f; w,z)\,d w \bullet z\\
\]

Then,

\[
\OmSum_{n=-\infty}^\infty f(n,z) \bullet z = \OmSum_{n=-\infty}^\infty \widehat{f}(n,z) \bullet z\\
\]

And,

\[
\MhSum_{n=-\infty}^\infty f(n,z) \bullet z = \MhSum_{n=-\infty}^\infty \widehat{f}(n,z) \bullet z\\
\]
\end{theorem}

Lastly, we place before the readers eyes a commutative diagram. This is intended to show the entire malleability of compositional analysis. We shall not make this a theorem with a proof. But instead, only hint at the depths to be uncovered.

\begin{center}
\begin{tikzcd}
f_j \arrow[dd, "\OmSum"'] \arrow[rr, "\widehat{\cdot}"] &  & \widehat{f_j} \arrow[dd, "\OmSum"] \\
                                                        &  &                                    \\
\OmSum f_j \arrow[rr, "\widehat{\cdot}"]                &  & \OmSum\widehat{f_j}               
\end{tikzcd}
\end{center}

Or, if one is more prone to the discussion being across $f[h]$ rather than $f$,

\begin{center}

\begin{tikzcd}
{f[h_j]} \arrow[dd, "\sum"'] \arrow[rr, "\widehat{\cdot}"] &  & {f[\widehat{h_j}]} \arrow[dd, "\sum"] \\
                                                           &  &                                       \\
{f[\sum h_j]} \arrow[rr, "\widehat{\cdot}"]                &  & {f[\sum \widehat{h_j}]}              
\end{tikzcd}
\end{center}

All of this is done assuming that $f_j$ is compactly normally convergent. Which is to say that,

\[
\sum_j ||f_j( w,z) - z||_{\mathcal{W}_a, \mathcal{K}} < \infty
\]

Where the supremum norm is taken across $ w \in \mathcal{W}_a$, and compact subsets $z \in \mathcal{K} \subset \mathbb{C}$. This is to say, the convergence is normal and good. Which, again, is for all intents and purposes equivalent to,

\[
\sum_j ||h_j( w,s)||_{\mathcal{W}_a, \gamma} < \infty
\]

Where the supremum norm is taken across $ w \in \mathcal{W}_a$, and $s \in \gamma$. This being a slightly more natural identification, though the former is no less correct; so long as we keep track of $f[h]$ and the fact they be defined by $\phi(s,z)$.

So to the man who said land ho,

\begin{theorem}[The Poor Man's Fourier's Linearity Theorem]\label{thmPoorLin}
Let $\mathcal{W}$ be a horizontal strip in the complex plane. Let $f_j( w,z) : \mathcal{W} \times \mathbb{C} \to \widehat{\mathbb{C}}$ be a sequence of Derived Residuals of moderate decay in $w$; and normally summable. Then if,

\[
\OmSum_j f_j( w,z) \bullet z = F( w,z)\\
\]

Then,
\[
\OmSum_j \widehat{f_j}(\xi,z) \bullet z = \widehat{F}(\xi,z)\\
\]
\end{theorem}

\chapter*{Appendix}
The author thought it might be enlightening to discuss further topics of more involved work; in the spirit of an advanced analysis. The compositional integral admits many similarities to the additive integral; so long as we pay close attention to where it isn't. Where it isn't similar is where the real exploits arise; so long as you can maintain a kind of correspondence between the additive integral and the compositional integral. We have only scratched the surface of what's possible.

The author has side-stepped most of the discussion of the compositional integral in real-analysis. Intrepidly the author jumped straight into complex analysis. He also felt much of the better results arise more naturally in complex analysis. In fact, he views the following as motivation that complex analysis should come first.

Suppose we were to take some function $H(\mathbf{x},s) : \mathbb{R}^n \times \mathcal{S} \to \widehat{\mathbb{C}}$. Where here, $H$ is meromorphic in $s$, but in $\mathbf{x}$, $H$ is merely $L^1$. Then, define the new derived residual,

\[
f(\mathbf{x},z) = \oint_{\gamma} H(\mathbf{x},s)\phi(s,z)\,ds\bullet z\\
\]

We go with the exact same identifications as before; $\phi$ is holomorphic and $\gamma$ is a Jordan curve. To this, we can define a Fourier Transform in $n$-variables; which produces a function $\widehat{f}(\mathbf{k},z) : \mathbb{R}^n \times \mathbb{C} \to \mathbb{C}$. Written,

\[
\widehat{f}(\mathbf{k},z) = \int_{\mathbb{R}^n} e^{-2 \pi i \mathbf{x} \cdot \mathbf{k}} \log(f;\mathbf{x},z)\,d\mathbf{x}\bullet z\\
\]

Where this is taken to mean,

\[
\widehat{f}(\mathbf{k},z) = \int_{-\infty}^\infty...\int_{-\infty}^\infty e^{-2 \pi i \mathbf{x} \cdot \mathbf{k}} \log(f;\mathbf{x},z)\,dx_1...dx_n\bullet z\\
\]

This is a perfectly fine function when we account for Fubini's theorem. Meaning, the above expressions reduce to,

\[
\oint_{\gamma} \widehat{H}(\mathbf{k},s)\phi(s,z)\,ds\bullet z\\
\]

The author has avoided a tremendous amount of discussion of real-analysis. He avoided it almost deliberately. Mostly he felt it may muddy the pond; and in a paper, it is difficult to have a foundational real-analytic discussion while you're having a foundational complex-analytic discussion. To that, we haven't really proven any real-analysis; though the author only begs the proof really; if you can do it in complex analysis it's not too far off to translate. But, digressions aside,

\[
f(\mathbf{x},z) = \int_{\mathbb{R}^n} e^{2\pi i \mathbf{x} \cdot \mathbf{k}} \log(\widehat{f};\mathbf{k},z)\,d\mathbf{k}\bullet z
\]

Now we require holomorphy in $s$ and $z$, but in $\mathbf{x}$ the author imagines it's as much as a free for all as one has in distributional analysis. \emph{Mais, N'on pas faire}. From this we can envision Poisson's summation across lattices; except now it's multiply layered infinite compositions. Which is to say,

\[
\OmSum_{\mathbf{a} \in \Lambda} f(\mathbf{a},z)\bullet z = \OmSum_{\mathbf{a}^* \in \Lambda^*} \widehat{f}(\mathbf{a}^*,z)\bullet z\\
\]

Where $\Lambda^*$ is the dual of $\Lambda$; speaking in lattices of course. We pretty much have just as much freedom as we have in additive analysis. When we concern ourselves with $\OmSum$ rather than $\sum$; and $\int\cdots ds\bullet z$ rather than $\int\cdots ds$; we have much of the same results. We just need to be careful of the context.

We could also define things like the Laplace transform; taking,

\[
f(x,z) = \oint_\gamma h(x,s)\phi(s,z)\,ds\bullet z\\
\]

And doing something like,

\[
F(y,z) = \int_0^\infty e^{-yx} \log(f;x,z)\,dx\bullet z\\
\]

Which satisfies,

\[
f(x,z) = \int_{c-i\infty}^{c+i\infty} \frac{e^{yx}}{2\pi i} \log(F;y,z)\,dy\bullet z\\
\]

We can begin to see the malleability of these transforms.

\section*{Closing Remarks}

We close this paper with a nod towards the possibilities of these expansions. The author feels this is only the surface value of these objects. This paper, however, had as its intent a slow analysis of $\int$ and $ds\bullet z$ and $\OmSum$, and the chemical interactions of these symbols. But init, it was done focusing on some bare rudimentary mechanics of complex analysis. We wanted to draw out the correspondence between $\int$ and $ds$ and $\sum$. There exists more cricks and crevices from this construction.

We hardly talked about the derived semi-group from the congruent integral. We made little to no mention of the dynamics of these things. We eschewed much of the questions regarding linear operators; and we glazed over much of the work. This paper, though; it was intended to motivate and high-light possibilities. It was only intended to break ground. There is still much digging to be done.

\end{document}